\documentclass{article}
\usepackage{graphicx} 

\usepackage[noadjust]{cite}
\usepackage{xcolor}
\usepackage{url}
\usepackage{latexsym,epsfig,amsfonts}
\usepackage{authblk}
\usepackage{tikz} 
\usepackage{verbatim}

\usepackage{amsthm}
\usepackage{amsmath}

\RequirePackage[all]{xy}

\newtheorem{thm}{Theorem}[section]

\newtheorem{prop}[thm]{Proposition}
\newtheorem{clm}[thm]{Claim}

\theoremstyle{definition}

\theoremstyle{remark}

\numberwithin{equation}{section}

\newcommand{\BibTeX}{B\kern-0.1emi\kern-0.017emb\kern-0.15em\TeX}
\newcommand{\XYpic}{$\mathrm{X\kern-0.3em\raisebox{-0.18em}{Y}}$-$\mathrm{pic}\,$}

\newcommand{\cl}{C \kern -0.1em \ell}  

\newcommand{\ed}{\end{document}}

\begin{document}

%
%
%
%
%
%
%
%
%

\title
 {Mutual-visibility and general position sets in Sierpi\'nski  triangle graphs}
\author[$\dagger$]{Danilo Kor\v ze}
%
\affil[$\dagger$]
{%
Faculty of Electrical Engineering and Computer Science\\
University of Maribor\\
Koro\v ska cesta 46\\
 SI-2000 Maribor, Slovenia }
%
%
\author[$\star$
]{Aleksander Vesel}
\affil[$\star$]{%
Faculty of Natural Sciences and Mathematics\\
University of Maribor\\
Koro\v ska cesta 160\\
 SI-2000 Maribor, Slovenia}

%
%
\date{\today}
\maketitle

\begin{abstract}
For a given graph \(G\), the general position problem asks for the largest set of vertices \(M \subseteq V(G)\) such that no three distinct vertices of \(M\) belong to a common shortest path in \(G\). A relaxation of this concept is based on the condition that two vertices \(x, y \in V(G)\) are \(M\)-visible, meaning there exists a shortest \(x, y\)-path in \(G\) that does not pass through any vertex of \(M \setminus \{x, y\}\). 

If every pair of vertices in \(M\) is \(M\)-visible, then \(M\) is called a mutual-visibility set of \(G\). The size of the largest mutual-visibility set of \(G\) is called the mutual-visibility number of \(G\). Some well-known variations of this concept consider the total, outer, and dual mutual-visibility sets of a graph.

We present results on the general position problem and the various mutual-visibility problems in Sierpiński triangle graphs.

\end{abstract}
\label{page:firstblob}

\section{Introduction and preliminaries}
Motivated by the well-known no-three-in-line problem, which challenges one to select the maximum number of points from an \( n \times n \) grid such that no three points are collinear \cite{Dudeney}, the idea of a general position set in a graph has been introduced \cite{Manuel, Chandran}.

This problem has incited studies of the largest general position sets in various families of graphs; for an extended list of references, see, for example, \cite{Klavzar}. Moreover, it has been observed that a similar concept, known as mutual-visibility sets, has been studied in a variety of contexts, including wireless sensor networks, mobile robot networks, and distributed computing \cite{Bhagat, Cicerone4, DiLuna, Poudel}. This relaxation of the general position problem, with the objective of maximizing the size of the largest mutual-visibility set—i.e., a set of vertices in a graph such that any two vertices in the set are mutually visible—was formally introduced in graph theory by Di Stefano \cite{DiS} in 2022. The problem is known to be NP-hard and has already been studied in various classes of graphs (see, for example, the references in \cite{vesel}).

Some results on visibility problems in graphs demonstrate that specific adjustments to visibility properties can be significant. Consequently, a range of novel mutual-visibility problems, including the total mutual-visibility problem, the dual mutual-visibility problem, and the outer mutual-visibility problem, have been proposed \cite{Cicerone} and later studied for several classes of graphs (see, for instance, \cite{Axenovich, Bujtas, kuziak, Tian}).

Inspired by the famous Sierpi\'nski triangle fractal, several intriguing families of graphs have been introduced, with the most prominent being the Sierpi\'nski graphs, the general Sierpi\'nski graphs, and the Sierpi\'nski triangle graphs. Interest in these graphs spans a variety of fields, including games like the Chinese Rings and the Tower of Hanoi, as well as areas such as topology, physics, and the study of interconnection networks, among others. For an extensive survey and classification of these so-called Sierpi\'nski-type graphs, interested readers are referred to \cite{hinz}.

Several properties of Sierpi\'nski triangle graphs have been studied, including the chromatic number, domination number, and pebbling number \cite{Tegola}; labeling and coloring \cite{KlavzarG}; vertex-colorings, edge-colorings, and total-colorings \cite{Jakovac2}; a 2-parametric generalization \cite{Jakovac}; as well as packing coloring \cite{packing}.

In this paper, we report on research into mutual-visibility and general position in Sierpi\'nski triangles. In Section 2, we present definitions and results crucial for the remainder of the paper, as well as two mathematical optimization techniques used to provide upper bounds on mutual-visibility set problems in the graphs of interest. In Section 3, we present the mutual-visibility numbers and general position numbers of Sierpi\'nski triangle graphs. The paper concludes with Section 4, where we determine variations of the mutual-visibility number of Sierpi\'nski triangle graphs.

\section{Preliminaries}

\subsection{Basic definitions}

If $G = (V(G),E(G))$
is a graph, then $M \subseteq V (G)$ is a {\em general position set} if for every triple of vertices 
$u,v,w \in V(G)$ we have $d_G(u, v) \not = d_G(u,w) + d_G(w,v)$, where $d_G(u,v)$ (or simply $d(u,v)$) denotes 
the  length of a shortest $u,v$-path (a path between $u$ and $v$) in $G$. 
The cardinality of a largest general position set of in $G$ is called the {\em general position number} of $G$ and denoted as $gp(G)$. 

Let $G = (V(G),E(G))$ be a graph, $M \subseteq V(G)$ and $u,v \in V(G)$.
We say that a  $u,v$-path $P$ is {\em $M$-free}, 
 if $P$ does not contain a vertex of $M \setminus \{u, v \}$.
Vertices $u, v \in V(G)$ are {\em $M$-visible}  if $G$ admits an $M$-free shortest $u,v$-path.
 
Let $M \subseteq V(G)$ and $\overline M=V(G) \setminus M$. Then we say that $M$ is a

 \begin{itemize}
  \item \emph{mutual-visibility} set, if every $u, v \in M$ are $M$-visible,

 \item \emph{total mutual-visibility} set, if every $u, v \in V(G)$ are $M$-visible,

 \item \emph{outer mutual-visibility} set, if every $u, v \in M$ are $M$-visible, and every $u \in M$, $v \in \overline M$ are $M$-visible,

\item \emph{dual mutual-visibility} set, if every $u, v \in M$ are $M$-visible, and every $u, v \in \overline M$ are $M$-visible.
 \end{itemize}

The cardinality of a largest mutual-visibility set, a largest total mutual-visibility set, a largest outer mutual-visibility set, and a largest dual mutual-visibility set will be respectively denoted by $\mu(G)$, $\mu_t(G)$, $\mu_o(G)$, and $\mu_d(G)$. 
The corresponding graph invariants will be respectively called the {\em mutual-visibility number}, the {\em total mutual-visibility number}, the {\em outer mutual-visibility number}, and the {\em dual mutual-visibility number} of $G$. 

Since every general position set of $G$ is obviously a mutual-visibility set of $G$, it holds that $gp(G) \le \mu(G)$ for every graph $G$. 
Moreover, every variant of a mutual-visibility set is itself a mutual-visibility set, it follows 
that $\mu_t(G) \le \mu(G)$,  $\mu_o(G) \le \mu(G)$ and $\mu_d(G) \le \mu(G)$ for every graph $G$. 

The set of vertices lying on all shortest $u, v$-paths is called
the {\em interval} between $u$ and $v$ and denoted by $I_{G}(u, v)$ 
 We will also write $I(u, v)$ when $G$ will be clear from the context.


For a positive integer $n$  we will use the notation $[n] = \{1, 2, \ldots, n\}$ and $[n]_0 = \{0, 1,  \ldots, n - 1\}$.

\subsection{Sierpi\'nski  triangle graphs}
The  {\em base-3 Sierpi\'{n}ski graphs $S^n$}  are defined such that we start with $S^0 = K_1$. 
For $n \ge 1$, the vertex set of $S^n$ is $[3]_0^n$ and the edge set is defined recursively as

$$E(S^n) = \{\{is, i t\} : i \in [3], \{s, t\} \in E(S^{n-1})\} \cup
\{\{i j^{n-1}, ji^{n-1} \} \, | \,  i, j \in [3], i\not= j \}.$$ 

Obviously, for $n\ge 1$,  $S^n$ can be constructed from three copies of  $S^{n-1}$. In left part of Fig. \ref{trans}, for example,  we can see  $S^2$ which 
consists of three copies of $S^1$.

For $n\ge 1$, the base-3 Sierpi\'{n}ski graphs $S^n$ belong 
to  a larger class of graphs that are known   
seen as generalized Sierpi\'{n}ski graphs  where $G=K_3$ \cite{hinz}. 

Let $n$ be a nonnegative integer. The class of the {\em Sierpi\'nski triangle graphs} $ST_3^n$ is obtained from 
$S^{n+1}$ by contracting all  non--clique edges. 
This transformation can be observed in Fig. \ref{trans}, where 
the graphs $S^2$ and $ST_3^1$ are depicted. 
Note that every Sierpi\'nski triangle graph $ST_3^n$ contains exactly three 
vertices of degree 2 that are called the \textit{extreme vertices} of  $ST_3^n$ and denoted $X(ST_3^n)$.
All other vertices of  $ST_3^n$ are called the \textit{non-extreme vertices}
of  $ST_3^n$.

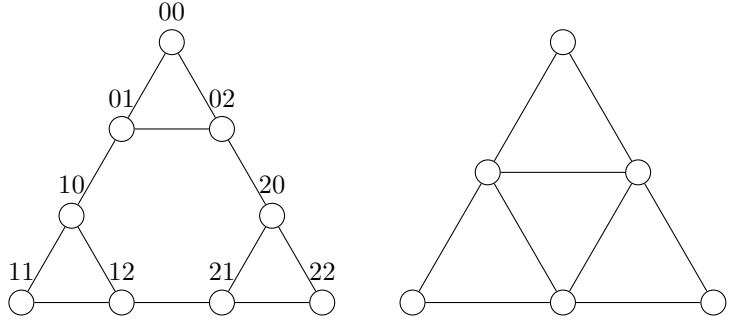
\begin{figure}[bt] 
\centering

\begin{tikzpicture}[scale=0.4]

\node[black, circle, draw] (p1) at (5/3*3,3*8.66/3) [label=above:$00$] {};
\node[black, circle, draw] (p2) at (5/3*2,2*8.66/3) [label=above:$01$] {};
\node[black, circle, draw] (p3) at (5/3*4,2*8.66/3) [label=above:$02$]  {};
\node[black, circle, draw] (p4) at (5/3*1,1*8.66/3) [label=above:$10$] {};
\node[black, circle, draw] (p5) at (0,0) [label=above:$11$] {};
\node[black, circle, draw] (p6) at (5/3*2,0) [label=above:$12$]  {};
\node[black, circle, draw] (p7) at (5/3*5,1*8.66/3) [label=above:$20$] {};
\node[black, circle, draw] (p8) at (5/3*4,0) [label=above:$21$] {};
\node[black, circle, draw] (p9) at (5/3*6,0) [label=above:$22$]  {};

\draw   (p1) -- (p2) -- (p3) -- (p1);
\draw   (p4) -- (p5) -- (p6) -- (p4);
\draw   (p7) -- (p8) -- (p9) -- (p7);
\draw   (p2) -- (p4);
\draw   (p3) -- (p7);
\draw   (p6) -- (p8);

\node[black, circle, draw] (q1) at (8+5/2*4,2*8.66/2)  {};
\node[black, circle, draw] (q2) at (8+5/2*3,1*8.66/2)  {};
\node[black, circle, draw] (q3) at (8+5/2*5,1*8.66/2)  {};
\node[black, circle, draw] (q4) at (8+5/2*2,0)  {};
\node[black, circle, draw] (q5) at (8+5/2*4,0)  {};
\node[black, circle, draw] (q6) at (8+5/2*6,0)  {};

\draw   (q1) -- (q2) -- (q3) -- (q1);
\draw   (q2) -- (q4) -- (q5) -- (q2);
\draw   (q3) -- (q5) -- (q6) -- (q3);
\end{tikzpicture}
\caption{Base-3 Sierpi\'{n}ski graph $S^2$ (left) and 
Sierpi\'nski triangle graph $ST_3^1$ (right)} \label{trans}
\end{figure}

There are various other definitions of Sierpi\'nski triangle graphs, all based on the fact that their drawings in the plane represent approximations of the Sierpi\'nski triangle fractal (see \cite{hinz}). More intuitively, Sierpi\'nski triangle graphs can be constructed iteratively. 

We start with a complete graph on 3 vertices, i.e., $ST_3^0$ is the triangle $K_3$. Now, assume that $ST_3^{n}$ is already constructed. To form $ST_3^{n+1}$, it is composed of three copies of $ST_3^{n}$, arranged in a way illustrated in Fig. \ref{properfig}, where $ST_3^2$ is composed of three copies of $ST_3^1$. Note that an extreme vertex of one copy of $ST_3^1$ is identified with an extreme vertex of another copy (this procedure is performed for exactly two extreme vertices of each copy).

Distances in Sierpi\'nski triangle graphs were studied in \cite{zemlja}.
We will need the following result. 

\begin{prop} \label{razdalje}
Let $n \ge 0$ be an integer. If $u,v \in V(ST_3^{n})$, then 
$d(u,v) \le 2^n$. Moreover, if $u$ and $v$ are extreme vertices of $ST_3^{n}$, then $d(u,v) = 2^n$.
\end{prop}

Let $n> m \ge k \ge 0$ be integers. We say that a copy of $ST_3^m$ 
and a copy of $ST_3^k$ in $ST_3^n$ are 
\textit{adjacent} if they have exactly one vertex in common. 

\begin{prop} \label{razdalje2}
Let $n \ge 2$ be an integer.  
If $H$ is a copy of $ST_3^{n-1}$ and $H'$ a copy of $ST_3^{n-2}$ 
such that $V(H) \cap V(H') = \{p\}$,
then $X(H) \cap I(u,v) = \{p\}$ for every $u \in V(H')\setminus X(H')$ and every $v \in V(H)$.
\end{prop}

\begin{proof}
Let ${\hat H}$ and ${\tilde H}$ be two other copies of $ST_3^{n-1}$ in $ST_3^{n}$, 
with ${\hat H}$  contains  $H'$.
Let $V(H) \cap V({\tilde H}) = \{ q \}$,  
$V({\hat H}) \cap V({\tilde H}) = \{ r \}$ and let $z$ denote the extreme 
vertex of $H'$ closest to $r$. 
Note that every path between $u \in V(H')\setminus X(H')$ and $v \in V(H)$ contains 
either $p$ or $q$. Clearly,  a shortest $u,r$-path contains  $z$,  
while a shortest $z,q$-path contains  either $r$ of $p$. 
By Proposition \ref{razdalje}, we have
$d(z,q)=d(z,p)+d(p,q)=d(z,r)+d(r,q) = 2^{n-1}+2^{n-2}$. 
Moreover, 
a length of a $u,v$-path that contains $q$ is at least 
$d(u,z)+d(z,r)+d(r,q)$, where $d(u,z) \ge 1$. Thus, the length of this path is more then  $2^{n-1}+2^{n-2}$.
On the other hand,   
the length of a $u,v$-path that contains $p$ is at most 
$d(z,p)+d(p,q) = 2^{n-1}+2^{n-2}$.
Since $d(u,v) = d(u,p) + d(p,v)$, we may conclude that
a shortest $u,v$-path cannot contain $q$.
\end{proof}

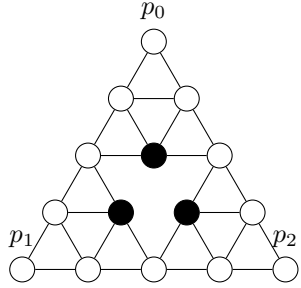
\begin{figure}[bt] 
\centering

\begin{tikzpicture}[scale=0.35]

\node[black, circle, draw] (p1) at (5/4*4,4*8.66/4) [label=above:$p_0$] {};
\node[black, circle, draw] (p2) at (5/4*3,3*8.66/4)  {};
\node[black, circle, draw] (p3) at (5/4*5,3*8.66/4)  {};
\node[black, circle, draw] (p4) at (5/4*2,2*8.66/4)  {};
\node[black, fill=black, circle, draw] (p5) at (5/4*4,2*8.66/4)  {};
\node[black, circle, draw] (p6) at (5/4*6,2*8.66/4)  {};
\node[black, circle, draw] (p7) at (5/4*1,1*8.66/4)  {};
\node[black, fill=black, circle, draw] (p8) at (5/4*3,1*8.66/4)  {};
\node[black, circle, draw] (p9) at (0,0) [label=above:$p_1$] {};
\node[black, circle, draw] (p10) at (5/4*2,0)  {};
\node[black, fill=black, circle, draw] (p11) at (5/4*5,1*8.66/4)  {};
\node[black, circle, draw] (p12) at (5/4*7,1*8.66/4)  {};
\node[black, circle, draw] (p13) at (5/4*4,0)  {};
\node[black, circle, draw] (p14) at (5/4*6,0)  {};
\node[black, circle, draw] (p15) at (5/4*8,0) [label=above:$p_2$] {};

\draw   (p1) -- (p2) -- (p3) -- (p1);
\draw   (p2) -- (p4) -- (p5) -- (p2);
\draw   (p3) -- (p5) -- (p6) -- (p3);
\draw   (p4) -- (p7) -- (p8) -- (p4);
\draw   (p7) -- (p9) -- (p10) -- (p7);
\draw   (p8) -- (p10) -- (p13) -- (p8);
\draw   (p6) -- (p11) -- (p12) -- (p6);
\draw   (p11) -- (p13) -- (p14) -- (p11);
\draw   (p12) -- (p14) -- (p15) -- (p12);

\end{tikzpicture}

\caption{Proper vertices of $ST_3^2$ } \label{properfig}
\end{figure}

The following sections will demonstrate that the mutual-visibility and general position sets of \(ST_3^2\) play a crucial role in understanding these sets within Sierpi\'nski triangle graphs.

Let ${\cal H}_2^n$ denote the set of all copies of $ST_3^2$ in $ST_3^n$.
Note that for $n\ge 2$, the cardinality of ${\cal H}_2^n$ is exactly $3^{n-2}$. 
Let $p_0$, $p_1$ and $p_2$ be the extreme vertices of $ST_3^2$. 
We observe that $ST_3^2$ admits exactly three vertices that do not belong to 
a shortest $p_0,p_1$-path, shortest $p_0,p_2$-path or shortest $p_1,p_2$-path, i.e.  
$|V(ST_3^2)\setminus (I(p_0,p_1) \cup I(p_0,p_2)) \cup I(p_1,p_2)| = 3$.
We refer to the vertices in the set  $V(ST_3^2)\setminus (I(p_0,p_1) \cup I(p_0,p_2)) \cup I(p_1,p_2)$ as the \textit{proper vertices} of $ST_3^2$.  
See Fig. \ref{properfig} for an illustration. We denote the set of all proper vertices of the graph $H_2$, which is isomorphic to $ST_3^2$, by $P_{H_2}$.

\subsection{Mathematical optimization methods}

We applied two well-known techniques to compute the mutual-visibility sets of hypercubes: Integer Linear Programming (ILP) and a reduction to SAT. Both methods have been successfully used in the past for distance-constrained problems on various finite and infinite graphs. Since these techniques, when adapted for the general position problem and various mutual-visibility problems, are already discussed in detail in \cite{gp} and \cite{vesel}, we present here only the main idea behind their application to the mutual-visibility problem.

Let  $G=(V, E)$ be a graph.
For a vertex $v \in V(G)$ we introduce the Boolean variable $x_v$ such that $x_v=1$ if and only if $v$ belongs to the mutual-visibility set $M$ of $G$.
The problem of finding the maximal mutual-visibility set can be formulated as an integer linear program defined by:\\

\begin{equation}
\textrm{maximize } \sum_{v \in V(G)} x_v
\end{equation}
subject to:
\begin{eqnarray}
 x_u + x_v - \sum_{P \in {\cal P}(u,v)}{z_{u,v,P}} \leq 1,  & \forall u,v \in V(G); \label{lpeq1} \\
 z_{u,v,P} + x_{z} \leq 1, & \forall u,v \in V(G); \label{lpeq2} \\
  & \forall P \in {\cal P}(u,v); \forall z \in V(P) \setminus \{u,v\}. \nonumber  
\end{eqnarray}

where ${\cal P}(u,v)$ denotes the set of all different shortest paths between vertices $u$ and $v$, while for $P \in {\cal P}(u,v)$ the set 
$V(P)\setminus\{u,v\}$ comprises all
intermediary vertices in the corresponding shortest path $P$ between $u$ and $v$.

 The additional variable $z_{u,v,P}$ equals 1 if and only if the  shortest path $P$ between $u$ and $v$ enables mutual visibility of these two vertices. 
 
 We applied  Gurobi Optimization \cite{gurobi} for solving ILP models.

Let  $G=(V, E)$ be a graph with $n$ vertices and $\ell$ a positive integer 
(the size of a potential mutual-visibility set).
To define a corresponding propositional satisfiability test (SAT), 
for every $v \in V(G)$ we introduce an atom $x_v$.
Intuitively, this atom expresses whether vertex $v$ is inside the mutual-visibility set $M$ or not. 
More precisely, $x_v=0$ if and only if $v$ belongs to the mutual-visibility set $M$.

First collection of propositional formulas define an encoding for cardinality constraints (known as $\ge k(x_1,...,x_n)$ constraints) which ensure that at most $k=n-\ell$ atoms are assigned value 1.  
The details of the applied encoding are presented in 
 \cite{sinz}. 
 
In order to define mutual-visibility constraints, 
consider the following propositional formulas:

\begin{eqnarray}
x_u \vee  x_v \vee \bigvee_{P \in {\cal P}(u,v)} \left[ \bigwedge_{ x_z \in V(P) \setminus \{u,v\}} x_{z} \right] \  (\forall u,v  \in V(G)),  
\label{sateq}
\end{eqnarray}

where ${\cal P}(u,v)$ denotes the set of all different shortest paths between vertices $u$ and $v$, while for $P \in {\cal P}(u,v)$ the set 
$V(P)\setminus\{u,v\}$ comprises all
intermediary vertices in the corresponding shortest path $P$ between $u$ and $v$.

Before using the above formulas  in a SAT solver, they have to be transformed to the conjunctive normal form.
These  propositional formulas transform a mutual-visibility problem into a propositional satisfiability test.

\section{Mutual visibility and general position}

\begin{prop} \label{edenM}
Let $n \ge 3$ be an integer and $M$  a mutual-visibility set of $ST_3^n$. 
If $H$ is a copy of $ST_3^{n-1}$ and $H'$ a copy of $ST_3^{n-2}$ 
adjacent to $H$ in $ST_3^n$ 
such that $V(H) \cap V(H') = \{p\}$ and 
$M \cap (V(H')\setminus X(H')) \not = \emptyset$, then 
$(M \cap V(H))  \cup  \{p \}$ is a mutual visibility set of $H$
\end{prop}

\begin{proof}
Let $u \in M \cap V(H')$. Since $M$ is a mutual  visibility set of $ST_3^n$, 
there exists an $M$-free shortest $u,v$-path in $ST_3^n$ for every $v \in M \cap V(H)$.    
Moreover, by Proposition \ref{razdalje2}, 
every shortest $u,v$-path in $ST_3^n$ contains $p$. Thus,  
there exists an $M$-free shortest $p,v$-path in $ST_3^n$ for every $v \in M \cap V(H)$.  It follows that $(M \cap (V(H')\setminus X(H'))  \cup  \{p\}$ is  a mutual visibility set of $H$. 
\end{proof}

Note that if the conditions of the above proposition are fulfilled,  it follows that 
$|(M \cap (V(H)\setminus X(H))| \le \mu(ST_3^{n-1}) - 1$.

\begin{figure}[bt] 
\centering

\begin{tikzpicture}[scale=0.35]

\node[black, fill=black, circle, draw] (p1) at (5/4*4,4*8.66/4)  {};
\node[black, circle, draw] (p2) at (5/4*3,3*8.66/4)  {};
\node[black, fill=black, circle, draw] (p3) at (5/4*5,3*8.66/4)  {};
\node[black, fill=black, circle, draw] (p4) at (5/4*2,2*8.66/4)  {};
\node[black, fill=black, circle, draw] (p5) at (5/4*4,2*8.66/4)  {};
\node[black, circle, draw] (p6) at (5/4*6,2*8.66/4)  {};
\node[black, circle, draw] (p7) at (5/4*1,1*8.66/4)  {};
\node[black,  circle, draw] (p8) at (5/4*3,1*8.66/4)  {};
\node[black, circle, draw] (p9) at (0,0)  {};
\node[black, circle, draw] (p10) at (5/4*2,0)  {};
\node[black, circle, draw] (p11) at (5/4*5,1*8.66/4)  {};
\node[black, circle, draw] (p12) at (5/4*7,1*8.66/4)  {};
\node[black, circle, draw] (p13) at (5/4*4,0)  {};
\node[black, circle, draw] (p14) at (5/4*6,0)  {};
\node[black, circle, draw] (p15) at (5/4*8,0)  {};

\draw   (p1) -- (p2) -- (p3) -- (p1);
\draw   (p2) -- (p4) -- (p5) -- (p2);
\draw   (p3) -- (p5) -- (p6) -- (p3);
\draw   (p4) -- (p7) -- (p8) -- (p4);
\draw   (p7) -- (p9) -- (p10) -- (p7);
\draw   (p8) -- (p10) -- (p13) -- (p8);
\draw   (p6) -- (p11) -- (p12) -- (p6);
\draw   (p11) -- (p13) -- (p14) -- (p11);
\draw   (p12) -- (p14) -- (p15) -- (p12);

\node[black, circle, draw] (q1) at (12+5/4*4,4*8.66/4)  {};
\node[black, fill=black,circle, draw] (q2) at (12+5/4*3,3*8.66/4)  {};
\node[black, fill=black,circle, draw] (q3) at (12+5/4*5,3*8.66/4)  {};
\node[black, fill=black,circle, draw] (q4) at (12+5/4*2,2*8.66/4)  {};
\node[black, circle, draw] (q5) at (12+5/4*4,2*8.66/4)  {};
\node[black, fill=black,circle, draw] (q6) at (12+5/4*6,2*8.66/4)  {};
\node[black, circle, draw] (q7) at (12+5/4*1,1*8.66/4)  {};
\node[black, circle, draw] (q8) at (12+5/4*3,1*8.66/4)  {};
\node[black, circle, draw] (q9) at (12,0)  {};
\node[black, circle, draw] (q10) at (12+5/4*2,0)  {};
\node[black, circle, draw] (q11) at (12+5/4*5,1*8.66/4)  {};
\node[black, circle, draw] (q12) at (12+5/4*7,1*8.66/4)  {};
\node[black, circle, draw] (q13) at (12+5/4*4,0)  {};
\node[black, circle, draw] (q14) at (12+5/4*6,0)  {};
\node[black, circle, draw] (q15) at (12+5/4*8,0)  {};

\draw   (q1) -- (q2) -- (q3) -- (q1);
\draw   (q2) -- (q4) -- (q5) -- (q2);
\draw   (q3) -- (q5) -- (q6) -- (q3);
\draw   (q4) -- (q7) -- (q8) -- (q4);
\draw   (q7) -- (q9) -- (q10) -- (q7);
\draw   (q8) -- (q10) -- (q13) -- (q8);
\draw   (q6) -- (q11) -- (q12) -- (q6);
\draw   (q11) -- (q13) -- (q14) -- (q11);
\draw   (q12) -- (q14) -- (q15) -- (q12);

\end{tikzpicture}

\caption{Four mutually visible vertices of a copy of $ST_3^1$ in $ST_3^2$} \label{four}
\end{figure}
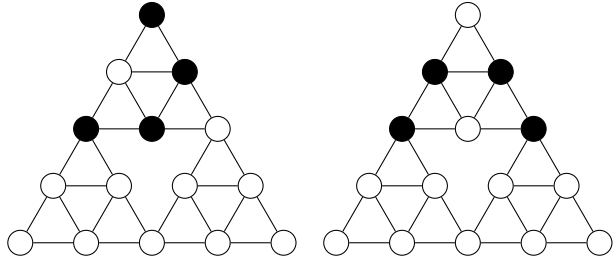

\begin{prop} \label{zgornjaM}
If $n\ge 1$, then $\mu(ST_3^n) \le 3^{n-1}+3$.
\end{prop}

\begin{proof}
It is not difficult to confirm the proposition for $ST_3^1$, since 
for a subset $M \subset V(ST_3^1)$ of cardinality 5 we can easily find 
a triple $u,v,w\in M$ such that $I(u,v)=\{u,v,w\}$.

Consider \( ST_3^2 \). Suppose \( M \) is a mutual-visibility set of \( ST_3^2 \) with 7 vertices.
Recall that \( ST_3^2 \) consists of 3 copies of \( ST_3^1 \), denoted \( H \), \( H' \), and \( H'' \). 

If \( |M \cap V(H)| = 4 \), then 
by symmetry, we can assume the configuration is as depicted in Fig. \ref{four}. It is straightforward to see that for every \( w \in V(ST_3^2) \setminus V(H) \), there always exist
\( u, v \in M \) such that \( v \) belongs to every shortest \( u, w \)-path. Thus, \( |M| = 4 \), which contradicts our assumption that \( |M| = 7 \).

Let \( M \cap V(H) = \{u, v, w\} \).
If \( u \) and \( v \) are both extreme vertices of \( H \) that also belong to another copy of \( ST_3^1 \), say \( H' \), then for every \( w \in V(ST_3^2) \setminus V(H) \), either \( u \) or \( v \) must lie on every shortest \( w, v \)-path. Therefore, \( |M| \le 3 \), which contradicts the assumption that \( |M| = 7 \).
If \( u \) belongs to \( H' \), but \( v \) does not, then
we can find a vertex \( v \in M \cap V(H) \) such that for every \( z \in V(H') \setminus (V(H) \cup V(H'')) \), the vertex \( u \) belongs to every shortest \( v, z \)-path. Thus, \( V(H') \setminus (V(H) \cup V(H'')) \) does not contain any vertices of \( M \), and \( |M| \le 6 \), which is a contradiction.

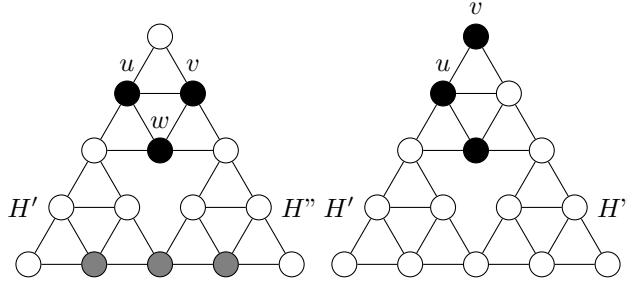
\begin{figure}[bt] 
\centering

\begin{tikzpicture}[scale=0.35]

\node[black, circle, draw] (p1) at (5/4*4,4*8.66/4)  {};
\node[black, fill=black,  circle, draw] (p2) at (5/4*3,3*8.66/4) [label=above:$u$] {};
\node[black, fill=black, circle, draw] (p3) at (5/4*5,3*8.66/4) [label=above:$v$]{};
\node[black,  circle, draw] (p4) at (5/4*2,2*8.66/4)  {};
\node[black, fill=black, circle, draw] (p5) at (5/4*4,2*8.66/4)  [label=above:$w$] {};
\node[black, circle, draw] (p6) at (5/4*6,2*8.66/4)  {};
\node[black, circle, draw] (p7) at (5/4*1,1*8.66/4) 
[label=left:$H'$] {};
\node[black,  circle, draw] (p8) at (5/4*3,1*8.66/4)  {};
\node[black, circle, draw] (p9) at (0,0)  {};
\node[black, fill=gray, circle, draw] (p10) at (5/4*2,0)  {};
\node[black, circle, draw] (p11) at (5/4*5,1*8.66/4)  {};
\node[black, circle, draw] (p12) at (5/4*7,1*8.66/4) [label=right:$H"$] {};
\node[black, fill=gray, circle, draw] (p13) at (5/4*4,0)  {};
\node[black, fill=gray, circle, draw] (p14) at (5/4*6,0)  {};
\node[black, circle, draw] (p15) at (5/4*8,0)  {};

\draw   (p1) -- (p2) -- (p3) -- (p1);
\draw   (p2) -- (p4) -- (p5) -- (p2);
\draw   (p3) -- (p5) -- (p6) -- (p3);
\draw   (p4) -- (p7) -- (p8) -- (p4);
\draw   (p7) -- (p9) -- (p10) -- (p7);
\draw   (p8) -- (p10) -- (p13) -- (p8);
\draw   (p6) -- (p11) -- (p12) -- (p6);
\draw   (p11) -- (p13) -- (p14) -- (p11);
\draw   (p12) -- (p14) -- (p15) -- (p12);

\node[black, fill=black,  circle, draw] (q1) at (12+5/4*4,4*8.66/4) [label=above:$v$] {};
\node[black, fill=black,circle, draw] (q2) at (12+5/4*3,3*8.66/4) [label=above:$u$] {};
\node[black, circle, draw] (q3) at (12+5/4*5,3*8.66/4)  {};
\node[black, circle, draw] (q4) at (12+5/4*2,2*8.66/4)  {};
\node[black, fill=black, circle, draw] (q5) at (12+5/4*4,2*8.66/4)  {};
\node[black, circle, draw] (q6) at (12+5/4*6,2*8.66/4)  {};
\node[black, circle, draw] (q7) at (12+5/4*1,1*8.66/4) [label=left:$H'$] {};
\node[black, circle, draw] (q8) at (12+5/4*3,1*8.66/4)  {};
\node[black, circle, draw] (q9) at (12,0)  {};
\node[black, circle, draw] (q10) at (12+5/4*2,0)  {};
\node[black, circle, draw] (q11) at (12+5/4*5,1*8.66/4)  {};
\node[black, circle, draw] (q12) at (12+5/4*7,1*8.66/4) [label=right:$H"$] {};
\node[black, circle, draw] (q13) at (12+5/4*4,0)  {};
\node[black, circle, draw] (q14) at (12+5/4*6,0)  {};
\node[black, circle, draw] (q15) at (12+5/4*8,0)  {};

\draw   (q1) -- (q2) -- (q3) -- (q1);
\draw   (q2) -- (q4) -- (q5) -- (q2);
\draw   (q3) -- (q5) -- (q6) -- (q3);
\draw   (q4) -- (q7) -- (q8) -- (q4);
\draw   (q7) -- (q9) -- (q10) -- (q7);
\draw   (q8) -- (q10) -- (q13) -- (q8);
\draw   (q6) -- (q11) -- (q12) -- (q6);
\draw   (q11) -- (q13) -- (q14) -- (q11);
\draw   (q12) -- (q14) -- (q15) -- (q12);

\end{tikzpicture}

\caption{Three mutually visible vertices of a copy of $ST_3^1$ in $ST_3^2$} \label{three}
\end{figure}

If \( \{u, v, w\} \) does not belong to \( H' \) and \( H'' \), then 
we consider the following cases. 
   
   If \( \{u, v, w\} \) induces a triangle  and  \( u \) is an extreme vertex of \( ST_3^2 \), then neither \( H' \) nor \( H'' \) can contain a vertex in \( M \), leading to a contradiction.

    If \( \{u, v, w\} \) induces a triangle  and  none of \( u, v, w \) are extreme vertices, refer to the left-hand side of Fig. \ref{three}. For every white vertex \( z \) of \( H' \) (resp. \( H'' \)), the shortest \( v, z \)-path (resp. \( u, z \)-path) contains either \( u \) or \( w \) (resp. either \( v \) or \( w \)). Thus, only gray vertices of \( H' \) and \( H'' \) can belong to \( M \), leading to \( |M| \le 6 \), which is a contradiction.

  If \( \{u, v, w\} \) does not induce a triangle, 
then, as depicted on the right-hand side of Fig. \ref{three}, for every vertex \( z \in V(H') \setminus V(H'') \), the shortest \( v, z \)-path contains \( u \). Therefore, vertices in \( V(H') \setminus V(H'') \) cannot be in \( M \), implying \( |M| \le 6 \), which contradicts our assumption.

Since we have reached a contradiction in every case where \( |M| = 7 \), we conclude that the cardinality of \( M \) is at most 6 when \( n = 2 \). 

For \( n = 3, 4 \), the proposition is confirmed by a computer search described in Subsection 2.3. Let \( n \ge 5 \). The rest of the proof proceeds by induction on \( n \). Assume the proposition holds for dimensions less than or equal to \( n-1 \), and let \( M \) be a mutual-visibility set of \( ST_3^n \). We consider the following cases.

A. \( ST_3^n \) contains a copy of \( ST_3^{n-1} \), say \( H \), such that for two copies of \( ST_3^{n-2} \) adjacent to \( H \), say \( H' \) and \( H'' \), neither \( V(H') \setminus X(H') \) nor \( V(H'') \setminus X(H'') \) contains a vertex in \( M \). 

Note that the subgraph of \( ST_3^n \) induced by \( V(ST_3^n) \setminus (V(H') \setminus X(H') \cup V(H'') \setminus X(H'')) \) contains one copy of \( ST_3^{n-1} \) and four copies of \( ST_3^{n-2} \). By the inductive hypothesis, we have
\[
|M| \le 3^{n-2} + 3 + 4 \cdot (3^{n-3} + 3) = 2 \cdot 3^{n-2} + 3^{n-3} + 15 < 3^{n-1} + 3.
\]

B. \( ST_3^n \) contains a copy of \( ST_3^{n-1} \), say \( H \), such that for two copies of \( ST_3^{n-2} \) adjacent to \( H \), say \( H' \) and \( H'' \), \( V(H') \setminus X(H') \) does not contain a vertex in \( M \), while \( V(H'') \setminus X(H'') \) contains a vertex in \( M \). 

By the inductive hypothesis, \( \mu(ST_3^{n-1}) \le 3^{n-3} + 3 \). Since \( V(H'') \setminus X(H'') \) contains a vertex in \( M \), by Proposition \ref{edenM}, we have
\[
|M \cap (V(H) \setminus X(H))| \le 3^{n-3} + 2.
\]

The subgraph of \( ST_3^n \) induced by \( V(ST_3^n) \setminus (V(H') \setminus X(H')) \) contains the following subgraphs: \( H \), another copy of \( ST_3^{n-1} \), and two copies of \( ST_3^{n-4} \). By the inductive hypothesis and the above discussion, we get
\[
|M| \le (3^{n-2} + 2) + (3^{n-2} + 3) + 2 \cdot (3^{n-3} + 3) = 2 \cdot 3^{n-2} + 2 \cdot 3^{n-3} + 11 < 3^{n-1} + 3.
\]

C.  Every copy of \( ST_3^{n-1} \) in \( ST_3^n \) admits two adjacent copies of \( ST_3^{n-2} \) that both contain a vertex in \( M \). 

By Proposition \ref{edenM} and the inductive hypothesis, we have
\[
|M| \le 3 \cdot (3^{n-2} + 1) = 3^{n-1} + 3.
\]

Since we have considered all possible cases, the proof is complete.
\end{proof}

We focus now on the graph \(ST_3^2\). Recall from Subsection 2.2 that every copy of \(ST_3^2\) in \(ST_3^n\) contains three proper vertices. For the set \(X(ST_3^2) = \{p_0, p_1, p_2\}\), these proper vertices form the set \(V(ST_3^2) \setminus (I(p_0, p_1) \cup I(p_0, p_2) \cup I(p_1, p_2))\).

Let $M \subseteq V(ST_3^n)$, $n\ge 2$ and $H_2$ a copy of $ST_3^2$ with extreme vertices 
$p_0$, $p_1$ and $p_2$ in $ST_3^n$. 
We say that $M$ is \textit{$H_2$-proper}  if 
for every $u\in V(H_2)\cap M$ and every $p_i$ and $p_j$, $i,j \in [3]_0$, 
it holds  that $u$ and $p_i$ are $M$-visible and 
$I(p_i,p_j)\cap M = \emptyset$.   

It is not difficult to confirm the following proposition.

\begin{prop} \label{proper}
Let $M \subseteq ST_3^n$, $n\ge 2$ and  $H_2$ be a copy of $ST_3^2$ in 
$ST_3^n$. Then $M$ is $H_2$-proper if and only if 
$M \cap V(H_2)  \subseteq \{ u, v, w \}$, where $\{ u, v, w \}$ is the set   of proper vertices of $H_2$.  
\end{prop}

\begin{figure}[bt] 
\centering

\begin{picture} (120, 120)

\put(0,0){\circle{1}}\put(7.5,0){\circle{1}}\put(15,0){\circle{1}}\put(3.75,6.49519){\circle{1}}\put(11.25,6.49519){\circle{1}}\put(7.5,12.9904){\circle{1}}\put(0.5,0){\line(1,0){6.5}}\put(8,0){\line(1,0){6.5}}\put(4.25,6.49519){\line(1,0){6.5}}\qbezier(0.25,0.45)(0.25,0.45)(3.5,6.04519)\qbezier(7.25,0.45)(7.25,0.45)(4,6.04519)\qbezier(7.75,0.45)(7.75,0.45)(11,6.04519)\qbezier(14.75,0.45)(14.75,0.45)(11.5,6.04519)\qbezier(4,6.94519)(4,6.94519)(7.25,12.5404)\qbezier(11,6.94519)(11,6.94519)(7.75,12.5404)
\put(15,0){\circle{1}}\put(22.5,0){\circle{1}}\put(30,0){\circle{1}}\put(18.75,6.49519){\circle{1}}\put(26.25,6.49519){\circle{1}}\put(22.5,12.9904){\circle{1}}\put(15.5,0){\line(1,0){6.5}}\put(23,0){\line(1,0){6.5}}\put(19.25,6.49519){\line(1,0){6.5}}\qbezier(15.25,0.45)(15.25,0.45)(18.5,6.04519)\qbezier(22.25,0.45)(22.25,0.45)(19,6.04519)\qbezier(22.75,0.45)(22.75,0.45)(26,6.04519)\qbezier(29.75,0.45)(29.75,0.45)(26.5,6.04519)\qbezier(19,6.94519)(19,6.94519)(22.25,12.5404)\qbezier(26,6.94519)(26,6.94519)(22.75,12.5404)
\put(7.5,12.9904){\circle{1}}\put(15,12.9904){\circle{1}}\put(22.5,12.9904){\circle{1}}\put(11.25,19.4856){\circle{1}}\put(18.75,19.4856){\circle{1}}\put(15,25.9808){\circle{1}}\put(8,12.9904){\line(1,0){6.5}}\put(15.5,12.9904){\line(1,0){6.5}}\put(11.75,19.4856){\line(1,0){6.5}}\qbezier(7.75,13.4404)(7.75,13.4404)(11,19.0356)\qbezier(14.75,13.4404)(14.75,13.4404)(11.5,19.0356)\qbezier(15.25,13.4404)(15.25,13.4404)(18.5,19.0356)\qbezier(22.25,13.4404)(22.25,13.4404)(19,19.0356)\qbezier(11.5,19.9356)(11.5,19.9356)(14.75,25.5308)\qbezier(18.5,19.9356)(18.5,19.9356)(15.25,25.5308)
\put(30,0){\circle{1}}\put(37.5,0){\circle{1}}\put(45,0){\circle{1}}\put(33.75,6.49519){\circle{1}}\put(41.25,6.49519){\circle{1}}\put(37.5,12.9904){\circle{1}}\put(30.5,0){\line(1,0){6.5}}\put(38,0){\line(1,0){6.5}}\put(34.25,6.49519){\line(1,0){6.5}}\qbezier(30.25,0.45)(30.25,0.45)(33.5,6.04519)\qbezier(37.25,0.45)(37.25,0.45)(34,6.04519)\qbezier(37.75,0.45)(37.75,0.45)(41,6.04519)\qbezier(44.75,0.45)(44.75,0.45)(41.5,6.04519)\qbezier(34,6.94519)(34,6.94519)(37.25,12.5404)\qbezier(41,6.94519)(41,6.94519)(37.75,12.5404)
\put(45,0){\circle{1}}\put(52.5,0){\circle{1}}\put(60,0){\circle{1}}\put(48.75,6.49519){\circle{1}}\put(56.25,6.49519){\circle{1}}\put(52.5,12.9904){\circle{1}}\put(45.5,0){\line(1,0){6.5}}\put(53,0){\line(1,0){6.5}}\put(49.25,6.49519){\line(1,0){6.5}}\qbezier(45.25,0.45)(45.25,0.45)(48.5,6.04519)\qbezier(52.25,0.45)(52.25,0.45)(49,6.04519)\qbezier(52.75,0.45)(52.75,0.45)(56,6.04519)\qbezier(59.75,0.45)(59.75,0.45)(56.5,6.04519)\qbezier(49,6.94519)(49,6.94519)(52.25,12.5404)\qbezier(56,6.94519)(56,6.94519)(52.75,12.5404)
\put(37.5,12.9904){\circle{1}}\put(45,12.9904){\circle{1}}\put(52.5,12.9904){\circle{1}}\put(41.25,19.4856){\circle{1}}\put(48.75,19.4856){\circle{1}}\put(45,25.9808){\circle{1}}\put(38,12.9904){\line(1,0){6.5}}\put(45.5,12.9904){\line(1,0){6.5}}\put(41.75,19.4856){\line(1,0){6.5}}\qbezier(37.75,13.4404)(37.75,13.4404)(41,19.0356)\qbezier(44.75,13.4404)(44.75,13.4404)(41.5,19.0356)\qbezier(45.25,13.4404)(45.25,13.4404)(48.5,19.0356)\qbezier(52.25,13.4404)(52.25,13.4404)(49,19.0356)\qbezier(41.5,19.9356)(41.5,19.9356)(44.75,25.5308)\qbezier(48.5,19.9356)(48.5,19.9356)(45.25,25.5308)
\put(15,25.9808){\circle{1}}\put(22.5,25.9808){\circle{1}}\put(30,25.9808){\circle{1}}\put(18.75,32.476){\circle{1}}\put(26.25,32.476){\circle{1}}\put(22.5,38.9711){\circle{1}}\put(15.5,25.9808){\line(1,0){6.5}}\put(23,25.9808){\line(1,0){6.5}}\put(19.25,32.476){\line(1,0){6.5}}\qbezier(15.25,26.4308)(15.25,26.4308)(18.5,32.026)\qbezier(22.25,26.4308)(22.25,26.4308)(19,32.026)\qbezier(22.75,26.4308)(22.75,26.4308)(26,32.026)\qbezier(29.75,26.4308)(29.75,26.4308)(26.5,32.026)\qbezier(19,32.926)(19,32.926)(22.25,38.5211)\qbezier(26,32.926)(26,32.926)(22.75,38.5211)
\put(30,25.9808){\circle{1}}\put(37.5,25.9808){\circle{1}}\put(45,25.9808){\circle{1}}\put(33.75,32.476){\circle{1}}\put(41.25,32.476){\circle{1}}\put(37.5,38.9711){\circle{1}}\put(30.5,25.9808){\line(1,0){6.5}}\put(38,25.9808){\line(1,0){6.5}}\put(34.25,32.476){\line(1,0){6.5}}\qbezier(30.25,26.4308)(30.25,26.4308)(33.5,32.026)\qbezier(37.25,26.4308)(37.25,26.4308)(34,32.026)\qbezier(37.75,26.4308)(37.75,26.4308)(41,32.026)\qbezier(44.75,26.4308)(44.75,26.4308)(41.5,32.026)\qbezier(34,32.926)(34,32.926)(37.25,38.5211)\qbezier(41,32.926)(41,32.926)(37.75,38.5211)
\put(22.5,38.9711){\circle{1}}\put(30,38.9711){\circle{1}}\put(37.5,38.9711){\circle{1}}\put(26.25,45.4663){\circle{1}}\put(33.75,45.4663){\circle{1}}\put(30,51.9615){\circle{1}}\put(23,38.9711){\line(1,0){6.5}}\put(30.5,38.9711){\line(1,0){6.5}}\put(26.75,45.4663){\line(1,0){6.5}}\qbezier(22.75,39.4211)(22.75,39.4211)(26,45.0163)\qbezier(29.75,39.4211)(29.75,39.4211)(26.5,45.0163)\qbezier(30.25,39.4211)(30.25,39.4211)(33.5,45.0163)\qbezier(37.25,39.4211)(37.25,39.4211)(34,45.0163)\qbezier(26.5,45.9163)(26.5,45.9163)(29.75,51.5115)\qbezier(33.5,45.9163)(33.5,45.9163)(30.25,51.5115)
\put(60,0){\circle{1}}\put(67.5,0){\circle{1}}\put(75,0){\circle{1}}\put(63.75,6.49519){\circle{1}}\put(71.25,6.49519){\circle{1}}\put(67.5,12.9904){\circle{1}}\put(60.5,0){\line(1,0){6.5}}\put(68,0){\line(1,0){6.5}}\put(64.25,6.49519){\line(1,0){6.5}}\qbezier(60.25,0.45)(60.25,0.45)(63.5,6.04519)\qbezier(67.25,0.45)(67.25,0.45)(64,6.04519)\qbezier(67.75,0.45)(67.75,0.45)(71,6.04519)\qbezier(74.75,0.45)(74.75,0.45)(71.5,6.04519)\qbezier(64,6.94519)(64,6.94519)(67.25,12.5404)\qbezier(71,6.94519)(71,6.94519)(67.75,12.5404)
\put(75,0){\circle{1}}\put(82.5,0){\circle{1}}\put(90,0){\circle{1}}\put(78.75,6.49519){\circle{1}}\put(86.25,6.49519){\circle{1}}\put(82.5,12.9904){\circle{1}}\put(75.5,0){\line(1,0){6.5}}\put(83,0){\line(1,0){6.5}}\put(79.25,6.49519){\line(1,0){6.5}}\qbezier(75.25,0.45)(75.25,0.45)(78.5,6.04519)\qbezier(82.25,0.45)(82.25,0.45)(79,6.04519)\qbezier(82.75,0.45)(82.75,0.45)(86,6.04519)\qbezier(89.75,0.45)(89.75,0.45)(86.5,6.04519)\qbezier(79,6.94519)(79,6.94519)(82.25,12.5404)\qbezier(86,6.94519)(86,6.94519)(82.75,12.5404)
\put(67.5,12.9904){\circle{1}}\put(75,12.9904){\circle{1}}\put(82.5,12.9904){\circle{1}}\put(71.25,19.4856){\circle{1}}\put(78.75,19.4856){\circle{1}}\put(75,25.9808){\circle{1}}\put(68,12.9904){\line(1,0){6.5}}\put(75.5,12.9904){\line(1,0){6.5}}\put(71.75,19.4856){\line(1,0){6.5}}\qbezier(67.75,13.4404)(67.75,13.4404)(71,19.0356)\qbezier(74.75,13.4404)(74.75,13.4404)(71.5,19.0356)\qbezier(75.25,13.4404)(75.25,13.4404)(78.5,19.0356)\qbezier(82.25,13.4404)(82.25,13.4404)(79,19.0356)\qbezier(71.5,19.9356)(71.5,19.9356)(74.75,25.5308)\qbezier(78.5,19.9356)(78.5,19.9356)(75.25,25.5308)
\put(90,0){\circle{1}}\put(97.5,0){\circle{1}}\put(105,0){\circle{1}}\put(93.75,6.49519){\circle{1}}\put(101.25,6.49519){\circle{1}}\put(97.5,12.9904){\circle{1}}\put(90.5,0){\line(1,0){6.5}}\put(98,0){\line(1,0){6.5}}\put(94.25,6.49519){\line(1,0){6.5}}\qbezier(90.25,0.45)(90.25,0.45)(93.5,6.04519)\qbezier(97.25,0.45)(97.25,0.45)(94,6.04519)\qbezier(97.75,0.45)(97.75,0.45)(101,6.04519)\qbezier(104.75,0.45)(104.75,0.45)(101.5,6.04519)\qbezier(94,6.94519)(94,6.94519)(97.25,12.5404)\qbezier(101,6.94519)(101,6.94519)(97.75,12.5404)
\put(105,0){\circle{1}}\put(112.5,0){\circle{1}}\put(120,0){\circle{1}}\put(108.75,6.49519){\circle{1}}\put(116.25,6.49519){\circle{1}}\put(112.5,12.9904){\circle{1}}\put(105.5,0){\line(1,0){6.5}}\put(113,0){\line(1,0){6.5}}\put(109.25,6.49519){\line(1,0){6.5}}\qbezier(105.25,0.45)(105.25,0.45)(108.5,6.04519)\qbezier(112.25,0.45)(112.25,0.45)(109,6.04519)\qbezier(112.75,0.45)(112.75,0.45)(116,6.04519)\qbezier(119.75,0.45)(119.75,0.45)(116.5,6.04519)\qbezier(109,6.94519)(109,6.94519)(112.25,12.5404)\qbezier(116,6.94519)(116,6.94519)(112.75,12.5404)
\put(97.5,12.9904){\circle{1}}\put(105,12.9904){\circle{1}}\put(112.5,12.9904){\circle{1}}\put(101.25,19.4856){\circle{1}}\put(108.75,19.4856){\circle{1}}\put(105,25.9808){\circle{1}}\put(98,12.9904){\line(1,0){6.5}}\put(105.5,12.9904){\line(1,0){6.5}}\put(101.75,19.4856){\line(1,0){6.5}}\qbezier(97.75,13.4404)(97.75,13.4404)(101,19.0356)\qbezier(104.75,13.4404)(104.75,13.4404)(101.5,19.0356)\qbezier(105.25,13.4404)(105.25,13.4404)(108.5,19.0356)\qbezier(112.25,13.4404)(112.25,13.4404)(109,19.0356)\qbezier(101.5,19.9356)(101.5,19.9356)(104.75,25.5308)\qbezier(108.5,19.9356)(108.5,19.9356)(105.25,25.5308)
\put(75,25.9808){\circle{1}}\put(82.5,25.9808){\circle{1}}\put(90,25.9808){\circle{1}}\put(78.75,32.476){\circle{1}}\put(86.25,32.476){\circle{1}}\put(82.5,38.9711){\circle{1}}\put(75.5,25.9808){\line(1,0){6.5}}\put(83,25.9808){\line(1,0){6.5}}\put(79.25,32.476){\line(1,0){6.5}}\qbezier(75.25,26.4308)(75.25,26.4308)(78.5,32.026)\qbezier(82.25,26.4308)(82.25,26.4308)(79,32.026)\qbezier(82.75,26.4308)(82.75,26.4308)(86,32.026)\qbezier(89.75,26.4308)(89.75,26.4308)(86.5,32.026)\qbezier(79,32.926)(79,32.926)(82.25,38.5211)\qbezier(86,32.926)(86,32.926)(82.75,38.5211)
\put(90,25.9808){\circle{1}}\put(97.5,25.9808){\circle{1}}\put(105,25.9808){\circle{1}}\put(93.75,32.476){\circle{1}}\put(101.25,32.476){\circle{1}}\put(97.5,38.9711){\circle{1}}\put(90.5,25.9808){\line(1,0){6.5}}\put(98,25.9808){\line(1,0){6.5}}\put(94.25,32.476){\line(1,0){6.5}}\qbezier(90.25,26.4308)(90.25,26.4308)(93.5,32.026)\qbezier(97.25,26.4308)(97.25,26.4308)(94,32.026)\qbezier(97.75,26.4308)(97.75,26.4308)(101,32.026)\qbezier(104.75,26.4308)(104.75,26.4308)(101.5,32.026)\qbezier(94,32.926)(94,32.926)(97.25,38.5211)\qbezier(101,32.926)(101,32.926)(97.75,38.5211)
\put(82.5,38.9711){\circle{1}}\put(90,38.9711){\circle{1}}\put(97.5,38.9711){\circle{1}}\put(86.25,45.4663){\circle{1}}\put(93.75,45.4663){\circle{1}}\put(90,51.9615){\circle{1}}\put(83,38.9711){\line(1,0){6.5}}\put(90.5,38.9711){\line(1,0){6.5}}\put(86.75,45.4663){\line(1,0){6.5}}\qbezier(82.75,39.4211)(82.75,39.4211)(86,45.0163)\qbezier(89.75,39.4211)(89.75,39.4211)(86.5,45.0163)\qbezier(90.25,39.4211)(90.25,39.4211)(93.5,45.0163)\qbezier(97.25,39.4211)(97.25,39.4211)(94,45.0163)\qbezier(86.5,45.9163)(86.5,45.9163)(89.75,51.5115)\qbezier(93.5,45.9163)(93.5,45.9163)(90.25,51.5115)
\put(30,51.9615){\circle{1}}\put(37.5,51.9615){\circle{1}}\put(45,51.9615){\circle{1}}\put(33.75,58.4567){\circle{1}}\put(41.25,58.4567){\circle{1}}\put(37.5,64.9519){\circle{1}}\put(30.5,51.9615){\line(1,0){6.5}}\put(38,51.9615){\line(1,0){6.5}}\put(34.25,58.4567){\line(1,0){6.5}}\qbezier(30.25,52.4115)(30.25,52.4115)(33.5,58.0067)\qbezier(37.25,52.4115)(37.25,52.4115)(34,58.0067)\qbezier(37.75,52.4115)(37.75,52.4115)(41,58.0067)\qbezier(44.75,52.4115)(44.75,52.4115)(41.5,58.0067)\qbezier(34,58.9067)(34,58.9067)(37.25,64.5019)\qbezier(41,58.9067)(41,58.9067)(37.75,64.5019)
\put(45,51.9615){\circle{1}}\put(52.5,51.9615){\circle{1}}\put(60,51.9615){\circle{1}}\put(48.75,58.4567){\circle{1}}\put(56.25,58.4567){\circle{1}}\put(52.5,64.9519){\circle{1}}\put(45.5,51.9615){\line(1,0){6.5}}\put(53,51.9615){\line(1,0){6.5}}\put(49.25,58.4567){\line(1,0){6.5}}\qbezier(45.25,52.4115)(45.25,52.4115)(48.5,58.0067)\qbezier(52.25,52.4115)(52.25,52.4115)(49,58.0067)\qbezier(52.75,52.4115)(52.75,52.4115)(56,58.0067)\qbezier(59.75,52.4115)(59.75,52.4115)(56.5,58.0067)\qbezier(49,58.9067)(49,58.9067)(52.25,64.5019)\qbezier(56,58.9067)(56,58.9067)(52.75,64.5019)
\put(37.5,64.9519){\circle{1}}\put(45,64.9519){\circle{1}}\put(52.5,64.9519){\circle{1}}\put(41.25,71.4471){\circle{1}}\put(48.75,71.4471){\circle{1}}\put(45,77.9423){\circle{1}}\put(38,64.9519){\line(1,0){6.5}}\put(45.5,64.9519){\line(1,0){6.5}}\put(41.75,71.4471){\line(1,0){6.5}}\qbezier(37.75,65.4019)(37.75,65.4019)(41,70.9971)\qbezier(44.75,65.4019)(44.75,65.4019)(41.5,70.9971)\qbezier(45.25,65.4019)(45.25,65.4019)(48.5,70.9971)\qbezier(52.25,65.4019)(52.25,65.4019)(49,70.9971)\qbezier(41.5,71.8971)(41.5,71.8971)(44.75,77.4923)\qbezier(48.5,71.8971)(48.5,71.8971)(45.25,77.4923)
\put(60,51.9615){\circle{1}}\put(67.5,51.9615){\circle{1}}\put(75,51.9615){\circle{1}}\put(63.75,58.4567){\circle{1}}\put(71.25,58.4567){\circle{1}}\put(67.5,64.9519){\circle{1}}\put(60.5,51.9615){\line(1,0){6.5}}\put(68,51.9615){\line(1,0){6.5}}\put(64.25,58.4567){\line(1,0){6.5}}\qbezier(60.25,52.4115)(60.25,52.4115)(63.5,58.0067)\qbezier(67.25,52.4115)(67.25,52.4115)(64,58.0067)\qbezier(67.75,52.4115)(67.75,52.4115)(71,58.0067)\qbezier(74.75,52.4115)(74.75,52.4115)(71.5,58.0067)\qbezier(64,58.9067)(64,58.9067)(67.25,64.5019)\qbezier(71,58.9067)(71,58.9067)(67.75,64.5019)
\put(75,51.9615){\circle{1}}\put(82.5,51.9615){\circle{1}}\put(90,51.9615){\circle{1}}\put(78.75,58.4567){\circle{1}}\put(86.25,58.4567){\circle{1}}\put(82.5,64.9519){\circle{1}}\put(75.5,51.9615){\line(1,0){6.5}}\put(83,51.9615){\line(1,0){6.5}}\put(79.25,58.4567){\line(1,0){6.5}}\qbezier(75.25,52.4115)(75.25,52.4115)(78.5,58.0067)\qbezier(82.25,52.4115)(82.25,52.4115)(79,58.0067)\qbezier(82.75,52.4115)(82.75,52.4115)(86,58.0067)\qbezier(89.75,52.4115)(89.75,52.4115)(86.5,58.0067)\qbezier(79,58.9067)(79,58.9067)(82.25,64.5019)\qbezier(86,58.9067)(86,58.9067)(82.75,64.5019)
\put(67.5,64.9519){\circle{1}}\put(75,64.9519){\circle{1}}\put(82.5,64.9519){\circle{1}}\put(71.25,71.4471){\circle{1}}\put(78.75,71.4471){\circle{1}}\put(75,77.9423){\circle{1}}\put(68,64.9519){\line(1,0){6.5}}\put(75.5,64.9519){\line(1,0){6.5}}\put(71.75,71.4471){\line(1,0){6.5}}\qbezier(67.75,65.4019)(67.75,65.4019)(71,70.9971)\qbezier(74.75,65.4019)(74.75,65.4019)(71.5,70.9971)\qbezier(75.25,65.4019)(75.25,65.4019)(78.5,70.9971)\qbezier(82.25,65.4019)(82.25,65.4019)(79,70.9971)\qbezier(71.5,71.8971)(71.5,71.8971)(74.75,77.4923)\qbezier(78.5,71.8971)(78.5,71.8971)(75.25,77.4923)
\put(45,77.9423){\circle{1}}\put(52.5,77.9423){\circle{1}}\put(60,77.9423){\circle{1}}\put(48.75,84.4375){\circle{1}}\put(56.25,84.4375){\circle{1}}\put(52.5,90.9327){\circle{1}}\put(45.5,77.9423){\line(1,0){6.5}}\put(53,77.9423){\line(1,0){6.5}}\put(49.25,84.4375){\line(1,0){6.5}}\qbezier(45.25,78.3923)(45.25,78.3923)(48.5,83.9875)\qbezier(52.25,78.3923)(52.25,78.3923)(49,83.9875)\qbezier(52.75,78.3923)(52.75,78.3923)(56,83.9875)\qbezier(59.75,78.3923)(59.75,78.3923)(56.5,83.9875)\qbezier(49,84.8875)(49,84.8875)(52.25,90.4827)\qbezier(56,84.8875)(56,84.8875)(52.75,90.4827)
\put(60,77.9423){\circle{1}}\put(67.5,77.9423){\circle{1}}\put(75,77.9423){\circle{1}}\put(63.75,84.4375){\circle{1}}\put(71.25,84.4375){\circle{1}}\put(67.5,90.9327){\circle{1}}\put(60.5,77.9423){\line(1,0){6.5}}\put(68,77.9423){\line(1,0){6.5}}\put(64.25,84.4375){\line(1,0){6.5}}\qbezier(60.25,78.3923)(60.25,78.3923)(63.5,83.9875)\qbezier(67.25,78.3923)(67.25,78.3923)(64,83.9875)\qbezier(67.75,78.3923)(67.75,78.3923)(71,83.9875)\qbezier(74.75,78.3923)(74.75,78.3923)(71.5,83.9875)\qbezier(64,84.8875)(64,84.8875)(67.25,90.4827)\qbezier(71,84.8875)(71,84.8875)(67.75,90.4827)
\put(52.5,90.9327){\circle{1}}\put(60,90.9327){\circle{1}}\put(67.5,90.9327){\circle{1}}\put(56.25,97.4279){\circle{1}}\put(63.75,97.4279){\circle{1}}\put(60,103.923){\circle{1}}\put(53,90.9327){\line(1,0){6.5}}\put(60.5,90.9327){\line(1,0){6.5}}\put(56.75,97.4279){\line(1,0){6.5}}\qbezier(52.75,91.3827)(52.75,91.3827)(56,96.9779)\qbezier(59.75,91.3827)(59.75,91.3827)(56.5,96.9779)\qbezier(60.25,91.3827)(60.25,91.3827)(63.5,96.9779)\qbezier(67.25,91.3827)(67.25,91.3827)(64,96.9779)\qbezier(56.5,97.8779)(56.5,97.8779)(59.75,103.473)\qbezier(63.5,97.8779)(63.5,97.8779)(60.25,103.473)

\put(60,103.923){\circle*{2}}\put(60,90.9327){\circle*{2}}\put(56.25,84.4375){\circle*{2}}\put(63.75,84.4375){\circle*{2}}
\put(45,64.9519){\circle*{2}}\put(41.25,58.4567){\circle*{2}}\put(48.75,58.4567){\circle*{2}}
\put(75,64.9519){\circle*{2}}\put(71.25,58.4567){\circle*{2}}\put(78.75,58.4567){\circle*{2}}

\put(30,38.9711){\circle*{2}}\put(26.25,32.476){\circle*{2}}\put(33.75,32.476){\circle*{2}}
\put(15,12.9904){\circle*{2}}\put(11.25,6.49519){\circle*{2}}\put(0,0) {\circle*{2}}\put(18.75,6.49519){\circle*{2}}
\put(45,12.9904){\circle*{2}}\put(41.25,6.49519){\circle*{2}}\put(48.75,6.49519){\circle*{2}}

\put(90,38.9711){\circle*{2}}\put(86.25,32.476){\circle*{2}}\put(93.75,32.476){\circle*{2}}
\put(75,12.9904){\circle*{2}}\put(71.25,6.49519){\circle*{2}}\put(78.75,6.49519){\circle*{2}}
\put(105,12.9904){\circle*{2}}\put(101.25,6.49519){\circle*{2}}\put(108.75,6.49519){\circle*{2}}\put(120,0){\circle*{2}}

\end{picture}

\caption{A largest mutual visibility set and the largest 
general position set of
Sierpi\'nski triangle graph $ST_3^4$ } \label{mv4}
\end{figure}
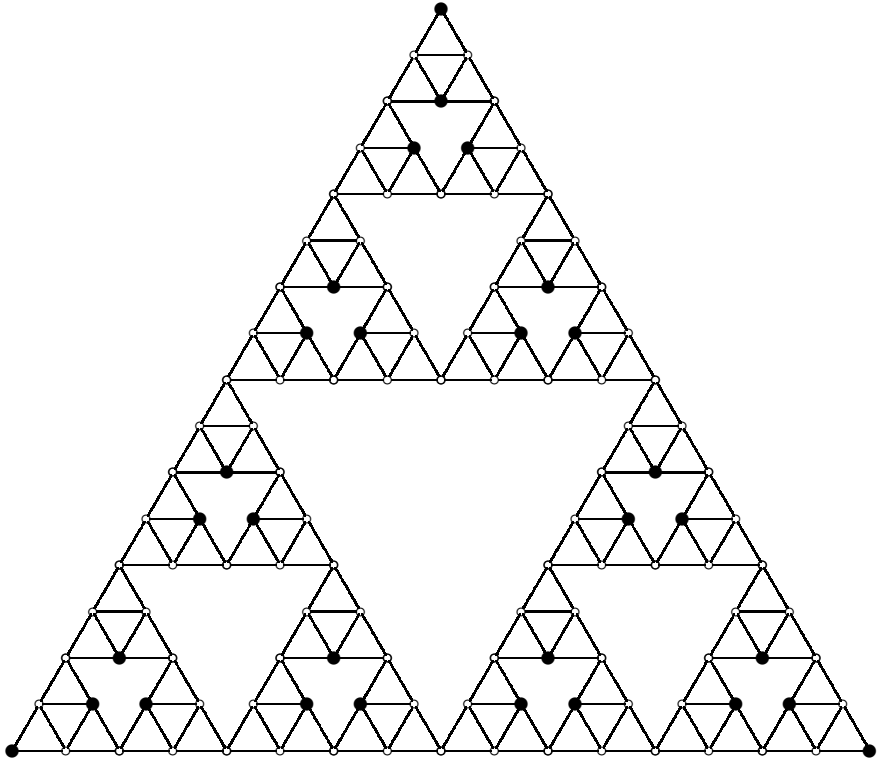

\begin{thm} \label{mutual}
If $n\ge 1$, then $\mu(ST_3^n) = 3^{n-1}+3$.
\end{thm}

\begin{proof}
Remind that we show in Proposition \ref{zgornjaM} that $\mu(ST_3^n) \le 3^{n-1}+3$. Hence,  
we have to show that there exists a mutual-visibility set of $ST_3^n$ with 
$3^{n-1}+3$ vertices for every $n\ge 1$. 
Since it is not difficult to construct  a mutual-visibility set with  
4 vertices of  $ST_3^1$ (see Fig. \ref{four}),  let assume that $n\ge 2$.
We now construct a mutual-visibility set $M$ of $ST_3^n$ as follows. 
For every copy of $ST_3^2$, say $H_2$, in $ST_3^n$  we insert all three proper vertices of $H_2$ in $M$ as well as all three extreme vertices of 
$ST_3^n$, i.e., 
$$M = (\bigcup_{H_2 \in {\cal H}_2^n}  P_{H_2}) \cup X(ST_3^n).$$ 


Clearly, $M$ contains exactly  $3^{n-1}+3$ vertices. 

We now demonstrate that \( M \) is a mutual-visibility set of \( ST_3^n \). Specifically, we need to show that every pair of vertices \( u, v \in M \) are \( M \)-visible in \( ST_3^n \).

   If \( u \) and \( v \) belong to the same copy of \( ST_3^2 \), the claim is straightforward. In this case, every pair of vertices in the same \( ST_3^2 \) copy is \( M \)-visible by the construction of \( M \).

   Suppose \( u \) and \( v \) belong to different copies of \( ST_3^2 \), say \( H_2 \) and \( H_2' \), respectively. If neither \( H_2 \) nor \( H_2' \) contains an extreme vertex of \( ST_3^n \), then \( M \) is \( H_2 \)- and \( H_2' \)-proper. This means that every vertex \( u \in M \cap V(H_2) \) and every vertex \( v \in M \cap V(H_2') \) are \( M \)-visible with respect to the extreme vertices \( p_i \) of \( H_2 \) and \( p_i' \) of \( H_2' \), respectively, where \( i \in [3]_0 \). Every shortest path between \( u \) and \( v \) must pass through some extreme vertex \( p_i \) of \( H_2 \) and some extreme vertex \( p_j' \) of \( H_2' \). Therefore, \( u \) and \( v \) are \( M \)-visible if 
   \( H_2 \) and \( H_2' \) are adjacent.

   If a shortest path between \( u \) and \( v \) crosses another copy of \( ST_3^2 \), say \( H_2'' \), let \( p_i'' \) be its extreme vertices. Since \( H_2'' \) is also proper, its extreme vertices \( p_k'' \) and \( p_{\ell}'' \) that the path crosses are \( M \)-visible. Thus, any path crossing \( H_2'' \) will still ensure that \( u \) and \( v \) are \( M \)-visible.

In conclusion, all cases confirm that \( u \) and \( v \) are \( M \)-visible, proving that \( M \setminus X(ST_3^n) \) is indeed a mutual-visibility set of \( ST_3^n \).

If $M\cap V(H_2)$ or $M\cap V(H_2')$ contains an extreme vertex of $ST_3^n$, then note that  
every $u,v$-path contains $p_i$ for some $i \in [3]_0$ and  $p_j'$ for some $j \in [3]_0$ such
that $p_i$ and $p_j'$ are not extreme vertices of $ST_3^n$. Moreover, 
$u$ and  $p_i$ (resp. $v$ and $p_j'$) are clearly $M$-visible even 
if $u$ (resp. $v$) is an extreme vertex of $ST_3^n$. 

Since we showed that  every $u,v  \in M$ are $M$-visible in $ST_3^n$, the proof is complete.
\end{proof}

The next result shows that a mutual-visibility set of 
$ST_3^n$, as defined in  the proof of Theorem \ref{mutual}, is also a 
general position set of $ST_3^n$.
A general position set of $G$ of cardinality $gp(G)$ is 
called a \textit{gp-set} of $G$.

\begin{thm}
If $n \ge 1$, then 
\begin{displaymath}
gp(ST_3^n) =
        \left \{ \begin{array}{llll}
              3,  &  n = 1   \\
              3^{n-1}+3,  &  {\rm otherwise}   \\
             \end{array}. \right.
\end{displaymath}
Moreover, if $n\ge 2$, then $M$ is gp-set of $ST_3^n$ if and only if
$M = (\bigcup_{H \in {\cal H}_2^n}  P_H)  \cup X(ST_3^n)$.
\end{thm}

\begin{proof}
Since $gp(ST_3^n) \le \mu(ST_3^n)$, the upper bounds are given by Theorem \ref{mutual}. 

Suppose first that $ST_3^1$ admits a general position set with 4 vertices. 
Since every mutual-visibility set is also a general position set, remind that $\mu(ST_3^1) = 4$. As already noted, mutual-visibility sets of 
$ST_3^1$ with 4 vertices are of the form depicted in Fig. \ref{four}. 
Since it is not difficult to see that a set of this form is not a 
general position set, we have that $gp(ST_3^1) \le 3$. The upper bound can be obtained by noting that all three extreme vertices of $ST_3^1$ clearly form its general position set. 

Let $n \ge 2$ and let us define the set $M$ as in the proof of Theorem \ref{mutual}, i.e. 
$M = (\bigcup_{H_2 \in {\cal H}_2^n}  P_{H_2}  \cup X(ST_3^n)$. 

We will show that $M$ is a general position set of $ST_3^n$,
i.e., we have to show that for
every pair of vertices
$u,v  \in M$, their interval $I(u,v)$ does not contain a vertex of 
$M$ apart of $u$ and $v$. 
If $u$ and $v$ belong to the same copy of $ST_3^2$, say $H_2$, it is not difficult to see that this claim holds. 
Moreover, if $p$ and $q$ are extreme vertices of $H_2$ and $u \in M \cap V(H_2)$, by the definition of the set $M$ we have $I(u,p)\cap (M\setminus \{u, p \}) = \emptyset$ and $I(p,q)\cap (M\setminus \{p, q \}) = \emptyset$. 

If \( u \) and \( v \) belong to different copies of \( ST_3^2 \), say \( H_2 \) and \( H_2' \) respectively, and \( H_2" \) denotes a copy of \( ST_3^2 \) that may intersect a shortest \( u,v \)-path, then note that in the proof of Theorem \ref{mutual}, it was shown that if a shortest \( u,v \)-path crosses \( H_2" \), then \( I(p_k", p_{\ell}") \cap (M \setminus \{p_k", p_{\ell}"\}) = \emptyset \), where \( p_k" \) and \( p_{\ell}" \) are the extreme vertices of \( H_2" \) that the shortest \( u,v \)-path crosses.

Since it follows that $I(u,v)\cap (M\setminus \{u, v \}) = \emptyset$, 
the proof of the first part of the theorem is complete.

To show that $(\bigcup_{H_2 \in {\cal H}_2^n}  P_{H_2})  \cup X(ST_3^n)$ is 
the unique general position set of the largest cardinality in  $ST_3^n$, we first examine $ST_3^2$ with a general position set $M$. 
Let $H$ be a copy of $ST_3^1$ in $ST_3^2$, 
$\{u\} = V(H) \cap P_{ST_3^2}$ and $\{p\} = X(H) \cap X(ST_3^2)$. 
We do not give all the details here, but it can be seen that if $M$ admits 
a vertex that does not belong to 
$\{u, p\}$, then $M$ possesses at most 5 vertices.  Thus, the claim for $ST_3^2$ holds.

For \(n = 3, 4\), the claim is confirmed by computer verification, showing that if a general position set \(M'\) of \(ST_3^n\) contains a vertex that does not belong to \(M\), then the cardinality of \(M'\) is at most \(3^{n-2} + 2\).

For \(n \ge 5\), the proof proceeds by induction on \(n\). As shown in the proof of Proposition \ref{zgornjaM}, if \(H\) is an arbitrary copy of \(ST_3^{n-1}\) in \(ST_3^n\), then \(ST_3^n\) admits a mutual-visibility set \(M\) of cardinality \(3^{n-1} + 3\) only if \(M \cap V(H)\) contains exactly \(3^{n-2} + 1\) vertices. Furthermore, if \(H'\) and \(H''\) are copies of \(ST_3^{n-2}\) adjacent to \(H\), then \(V(H') \cap M \neq \emptyset\) and \(V(H'') \cap M \neq \emptyset\).

Let \(V(H) \cap V(H') = \{p\}\) and \(V(H) \cap V(H'') = \{q\}\), and consider \(M\) as a general position set of cardinality \(3^{n-1} + 3\) in \(ST_3^n\). Since every general position set is also a mutual-visibility set, we can apply the proof of Proposition \ref{edenM} to deduce that \(M \cap V(H)\) contains \(3^{n-2} + 1\) vertices only if \((M \cap V(H)) \cup \{p, q\}\) forms a general position (and mutual-visibility) set in \(H\).

Thus, \(|(M \cap V(H)) \cup \{p, q\}| = 3^{n-1} + 3\), and by the inductive hypothesis, it follows that \((M \cap V(H)) \cup \{p, q\} = (\bigcup_{H_2 \in {\cal H}_2^{n-1}} P_{H_2}) \cup X(ST_3^{n-1})\).

Since the above conclusion holds for every copy of \(ST_3^{n-1}\), the proof is complete.
\end{proof}

\section{Varieties}
We first consider outer mutual-visibility sets of Sierpi\'nski  triangle graphs, which can be studied in a similar manner to "regular" mutual-visibility sets.

\begin{prop} \label{edenO}
Let $n \ge 3$ be an integer and $M$ an outer mutual-visibility set of $ST_3^n$. 
If $H$ is a copy of $ST_3^{n-1}$ and $H'$ a copy of $ST_3^{n-2}$ 
adjacent to $H$ in $ST_3^n$ such that 
$(M \cap V(H))  \cup  \{p \}$ and 
$M \cap (V(H')\setminus X(H')) \not = \emptyset$, then 
$(M \cap (V(H)\setminus X(H)))  \cup  \{p\}$ is  an outer mutual-visibility set of $H$
\end{prop}

\begin{proof}
Let $u \in M \cap V(H')$. Since $M$ is an outer mutual-visibility set of $ST_3^n$, 
there exists an $M$-free shortest $u,v$-path in $ST_3^n$ for every $v \in V(H)$.    
Moreover, since every shortest $u,v$-path in $ST_3^n$ contains $p$, 
there exists an $M$-free shortest $p,v$-path in $ST_3^n$ for every $v \in V(H)$.  It follows that $(M \cap (V(H)\setminus X(H)))  \cup  \{p\}$ is  an outer mutual-visibility set of $H$. 
\end{proof}

Note that if $M$  and $H$ satisfy the conditions of the above proposition, 
we have that $|(M \cap (V(H)\setminus X(H))| \le \mu_o(ST_3^{n-1}) - 1$. 

\begin{figure}[bt] 
\centering

\begin{tikzpicture}[scale=0.35]

\node[black, circle, draw] (p1) at (5/4*4,4*8.66/4)  {};
\node[black, fill=black,  circle, draw] (p2) at (5/4*3,3*8.66/4) [label=above:$v$] {};
\node[black, fill=black, circle, draw] (p3) at (5/4*5,3*8.66/4) {};
\node[black,  fill=black, circle, draw] (p4) at (5/4*2,2*8.66/4) [label=above:$u$] {};
\node[black, circle, draw] (p5) at (5/4*4,2*8.66/4)  {};
\node[black, circle, draw] (p6) at (5/4*6,2*8.66/4) {};
\node[black, circle, draw] (p7) at (5/4*1,1*8.66/4) 
[label=left:$w$] {};
\node[black,  circle, draw] (p8) at (5/4*3,1*8.66/4)  {};
\node[black, circle, draw] (p9) at (0,0)  {};
\node[black,  circle, draw] (p10) at (5/4*2,0)  {};
\node[black, circle, draw] (p11) at (5/4*5,1*8.66/4)  {};
\node[black, circle, draw] (p12) at (5/4*7,1*8.66/4)  {};
\node[black,  circle, draw] (p13) at (5/4*4,0)  {};
\node[black, fill=black, circle, draw] (p14) at (5/4*6,0)  {};
\node[black, circle, draw] (p15) at (5/4*8,0)  {};

\draw   (p1) -- (p2) -- (p3) -- (p1);
\draw   (p2) -- (p4) -- (p5) -- (p2);
\draw   (p3) -- (p5) -- (p6) -- (p3);
\draw   (p4) -- (p7) -- (p8) -- (p4);
\draw   (p7) -- (p9) -- (p10) -- (p7);
\draw   (p8) -- (p10) -- (p13) -- (p8);
\draw   (p6) -- (p11) -- (p12) -- (p6);
\draw   (p11) -- (p13) -- (p14) -- (p11);
\draw   (p12) -- (p14) -- (p15) -- (p12);

\node[black,   circle, draw] (q1) at (12+5/4*4,4*8.66/4)  {};
\node[black, fill=black, circle, draw] (q2) at (12+5/4*3,3*8.66/4) [label=above:$z_2$] {};
\node[black, fill=black,circle, draw] (q3) at (12+5/4*5,3*8.66/4) [label=above:$z_1$] {};
\node[black, circle, draw] (q4) at (12+5/4*2,2*8.66/4) [label=above:$w$] {};
\node[black, fill=black, circle, draw] (q5) at (12+5/4*4,2*8.66/4) [label=above:$u$] {};
\node[black, circle, draw] (q6) at (12+5/4*6,2*8.66/4)  {};
\node[black, circle, draw] (q7) at (12+5/4*1,1*8.66/4) [label=left:$H'$] {};
\node[black, fill=black, circle, draw] (q8) at (12+5/4*3,1*8.66/4) [label=above:$v$] {};
\node[black, circle, draw] (q9) at (12,0)  {};
\node[black, circle, draw] (q10) at (12+5/4*2,0)  {};
\node[black, circle, draw] (q11) at (12+5/4*5,1*8.66/4)  {};
\node[black, circle, draw] (q12) at (12+5/4*7,1*8.66/4) [label=right:$H"$] {};
\node[black, circle, draw] (q13) at (12+5/4*4,0)  {};
\node[black, circle, draw] (q14) at (12+5/4*6,0)  {};
\node[black, circle, draw] (q15) at (12+5/4*8,0)  {};

\draw   (q1) -- (q2) -- (q3) -- (q1);
\draw   (q2) -- (q4) -- (q5) -- (q2);
\draw   (q3) -- (q5) -- (q6) -- (q3);
\draw   (q4) -- (q7) -- (q8) -- (q4);
\draw   (q7) -- (q9) -- (q10) -- (q7);
\draw   (q8) -- (q10) -- (q13) -- (q8);
\draw   (q6) -- (q11) -- (q12) -- (q6);
\draw   (q11) -- (q13) -- (q14) -- (q11);
\draw   (q12) -- (q14) -- (q15) -- (q12);

\end{tikzpicture}

\caption{Cases in the proof of Proposition \ref{zgornjaO}  } \label{cases}
\end{figure}
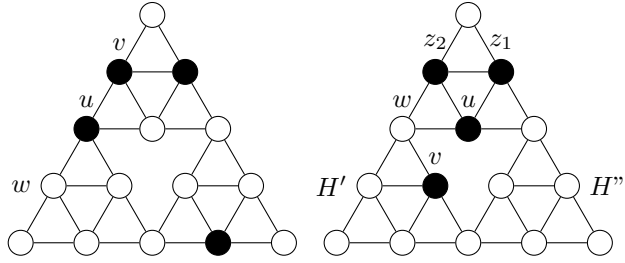

\begin{prop} \label{zgornjaO}
If $n \ge 1$, then 
\begin{displaymath}
\mu_o(ST_3^n) \le
        \left \{ \begin{array}{llll}
              3,  &  n = 1   \\
              3^{n-2}+3,  &  {\rm otherwise}   \\
             \end{array}. \right.
\end{displaymath}
\end{prop}

\begin{proof}
Let $M$ be an outer mutual-visibility set of $ST_3^n$. 
To confirm the proposition for $ST_3^1$, suppose that $M$ is of cardinality 4. Obviously, every outer mutual-visibility set is also a 
mutual-visibility set. Let observe all mutual-visibility sets with four 
vertices of a copy of $ST_3^1$, say $H$, in Fig. \ref{four}. It is not difficult to  find $u \in M$ and $v \in V(H)\setminus M$ such that the 
shortest $u,v$-path  contains a vertex of $M$. Since we obtains a contradiction, we have $\mu_o(ST_3^1) \le 3$.

Let $M$ be an outer 
mutual-visibility set of $ST_3^2$. Suppose to the contrary that $M$ 
possesses 5 vertices. Note first that $M$ cannot contain 4 non-proper vertices, since otherwise $M$ contains 2 vertices, say $u$ and $v$,  such that $u,v \in I(p_i,p_j)$ and $u \not \in \{p_i, p_j\}$, 
where $p_i$ and $p_j$ are extreme vertices of $ST_3^2$. 
Consequently,  there exists a vertex $w \in I(p_i,p_j)$ such that 
$u$ belongs to the shortest $v,w$-path (see the left hand side of Fig. \ref{cases}). 

Next we show that 
$M$ cannot contain 2 proper vertices. Suppose to the contrary that $u$ and $v$ are proper vertices and $u,v \in M$. Let $H$, $H'$ and $H"$ be distinct copies of $ST_3^1$ in $ST_3^2$ and let $u\in V(H)$ and $v\in V(H')$. 
Let $w$ be the vertex adjacent to both $u$ and $w$. 
For every vertex $z \in V(H")$, any
shortest $w,z$-path would necessarily pass through $u$ or $v$ (see the right side of Fig. \ref{cases}).
Consequently, none of the vertices in $V(H")$ can belong to $M$, implying that $M$ must include 3 vertices from $V(H)$ or $V(H')$.

Assume that \( M \) includes 3 vertices \( u, z_1, z_2 \) from \( V(H) \). If \( z_1 \) is an extreme vertex of \( ST_3^2 \), there exists a vertex \( w \) adjacent to \( z_2 \) such that \( z_2 \) lies on the shortest \( z_1, w \)-path. This implies that neither \( z_1 \) nor \( z_2 \) can be an extreme vertex of \( ST_3^2 \), as their inclusion would disrupt mutual-visibility. Additionally, if \( u, z_1, z_2 \) form a triangle, \( u \) and the extreme vertex adjacent to both \( z_1 \) and \( z_2 \) would not be mutually visible within \( M \).

Thus, we conclude that \( M \) cannot contain 5 vertices, confirming that the outer mutual-visibility set of \( ST_3^2 \) must have a cardinality less than 5. This assertion settles the case.

For $n=3,4,5$ the claim regarding the the cardinality of outer mutual-visibility sets  is confirmed by a computer.
Let $n \ge 6$. 
The remainder of the proof proceeds by induction on \( n \), similar to the approach used in the proof of Proposition \ref{zgornjaM}.
Let the claim holds for dimensions up to  $n-1$ and
let $M$ be an outer mutual-visibility set of $ST_3^n$. 

Suppose first that $ST_3^{n}$ admits  a copy of $ST_3^{n-1}$, say $H$,
such that for two adjacent copies of $ST_3^{n-2}$, say $H'$ and $H''$, neither $V(H')\setminus X(H')$ nor $V(H")\setminus X(H")$  
contains a vertex in $M$. Note that the subgraph of $ST_3^n$
induced by $V(ST_3^n)\setminus (V(H')\setminus X(H') \cup V(H")\setminus X(H")$  
contains one copy of $ST_3^{n-1}$ and four copies of $ST_3^{n-2}$. 
By the inductive hypothesis we have that $|M| \le 3^{n-3}+3 + 
4 \cdot (3^{n-4}+3) = 2 \cdot  3^{n-3} + 3^{n-4}  + 15 < 3^{n-2} + 3$.  

Suppose next that $ST_3^{n}$ admits  a copy of $ST_3^{n-1}$, say $H$,
and of the two adjacent copies of $ST_3^{n-2}$, say $H'$ and $H''$, only
$V(H")\setminus X(H")$ contains a vertex in $M$. 
Given 
$\mu_o(ST_3^{n-1}) \le 3^{n-3}+3$ and noting  $V(H")\setminus X(H")$ contains a vertex in $M$, by Proposition \ref{edenO} we have that 
$|M \cap V(H)\setminus X(H)| \le 3^{n-3}+2$. 

Note also that the subgraph of $ST_3^n$
induced by $V(ST_3^n)\setminus (V(H')\setminus X(H'))$  
contains $H$, a copy of $ST_3^{n-1}$  adjacent to $H$ and two copies of $ST_3^{n-4}$. 
By the inductive hypothesis and the above discussion we have that 
$|M| \le 3^{n-3}+2 3^{n-3}+ 3  + 2 \cdot (3^{n-4}+3) < 3^{n-2} + 3$.  

Finally, suppose that every copy of $ST_3^{n-1}$ in $ST_3^n$ 
admits two adjacent copies of $ST_3^{n-2}$ that both contain a vertex in 
$M$.
By Proposition \ref{edenO} and the inductive hypothesis we the have  
$|M| \le 3 \cdot (3^{n-3}+1) =  3^{n-2} + 3$. 

By considering all cases — whether \( ST_3^n \) contains copies of \( ST_3^{n-1} \) with one, none, or both adjacent \( ST_3^{n-2} \) copies containing vertices from \( M \) — we have shown that the outer mutual-visibility set cannot exceed \( 3^{n-2} + 3 \) in size. 
\end{proof}

We focus again on the graph \(ST_3^2\). Recall from Subsection 2.2 that every copy of \(ST_3^2\) in \(ST_3^n\) contains three proper vertices. For the set \(X(ST_3^2) = \{p_0, p_1, p_2\}\), these proper vertices form the set \(V(ST_3^2) \setminus (I(p_0, p_1) \cup I(p_0, p_2) \cup I(p_1, p_2))\).

Let $M \subseteq V(ST_3^n)$, $n\ge 2$ and $H_2$ a copy of $ST_3^2$. 
We say that $M$ is \textit{$H_2$-outer-proper}  if for every $u\in V(H_2)$ and every $p_i \in X(H_2)$, $i \in [3]_0$, 
it holds  that $u$ and $p_i$ are $M$-visible.   

\begin{prop} \label{properO}
Let $M \subseteq ST_3^n$, $n\ge 2$ and  $H$ be a copy of $ST_3^2$ in 
$ST_3^n$. Then $M$ is $H_2$-outer-proper if and only if  
$M \cap V(H_2) = \emptyset$ or $M \cap V(H_2) = \{u\}$, where $u$ is a proper 
vertex of $ST_3^2$.
\end{prop}

\begin{proof}
If $M \cap V(H_2) = \emptyset$, the claim is trivial. 
Suppose then that $M \cap V(H_2) \not = \emptyset$. 

If $M \cap V(H_2) = \{u\}$ and $u$ is a proper vertex of $ST_3^2$, it is 
not difficult to check that 
for every $v\in V(H_2)$ and every $i \in [3]_0$ it holds that $v$ and 
$p_i \in X(H_2)$ are $M$-visible.   

Let us assume that $M$ is $H$-outer-proper, yet the conditions of the proposition are not fulfilled. 
If $u \in M \cap V(H_2) $ and $u$ is not a proper vertex of $ST_3^2$, then, 
by the definition of a proper vertex, there exist $i,j \in [3]_0$ such that $u$ belongs to the shortest $p_i,p_j$-path and we obtain a contradiction.  
If $\{u,v\} \subseteq M \cap V(H_2)$, then $u$ and $v$ must be both 
proper by the above discussion. We may assume without loss of generality that 
$d(u,p_0)=d(v,p_1)=2$. Let $x$ be the vertex adjacent to both $u$ and $v$.
Since it is not difficult to see that neither of shortest 
 $x, p_2$-paths is $M$-free (one contains $u$ and the other $b$), we again obtain a contradiction. This assertion concludes the proof. 
\end{proof}

\begin{figure}[bt] 
\centering

\begin{picture} (120, 120)

\put(0,0){\circle{1}}\put(7.5,0){\circle{1}}\put(15,0){\circle{1}}\put(3.75,6.49519){\circle{1}}\put(11.25,6.49519){\circle{1}}\put(7.5,12.9904){\circle{1}}\put(0.5,0){\line(1,0){6.5}}\put(8,0){\line(1,0){6.5}}\put(4.25,6.49519){\line(1,0){6.5}}\qbezier(0.25,0.45)(0.25,0.45)(3.5,6.04519)\qbezier(7.25,0.45)(7.25,0.45)(4,6.04519)\qbezier(7.75,0.45)(7.75,0.45)(11,6.04519)\qbezier(14.75,0.45)(14.75,0.45)(11.5,6.04519)\qbezier(4,6.94519)(4,6.94519)(7.25,12.5404)\qbezier(11,6.94519)(11,6.94519)(7.75,12.5404)
\put(15,0){\circle{1}}\put(22.5,0){\circle{1}}\put(30,0){\circle{1}}\put(18.75,6.49519){\circle{1}}\put(26.25,6.49519){\circle{1}}\put(22.5,12.9904){\circle{1}}\put(15.5,0){\line(1,0){6.5}}\put(23,0){\line(1,0){6.5}}\put(19.25,6.49519){\line(1,0){6.5}}\qbezier(15.25,0.45)(15.25,0.45)(18.5,6.04519)\qbezier(22.25,0.45)(22.25,0.45)(19,6.04519)\qbezier(22.75,0.45)(22.75,0.45)(26,6.04519)\qbezier(29.75,0.45)(29.75,0.45)(26.5,6.04519)\qbezier(19,6.94519)(19,6.94519)(22.25,12.5404)\qbezier(26,6.94519)(26,6.94519)(22.75,12.5404)
\put(7.5,12.9904){\circle{1}}\put(15,12.9904){\circle{1}}\put(22.5,12.9904){\circle{1}}\put(11.25,19.4856){\circle{1}}\put(18.75,19.4856){\circle{1}}\put(15,25.9808){\circle{1}}\put(8,12.9904){\line(1,0){6.5}}\put(15.5,12.9904){\line(1,0){6.5}}\put(11.75,19.4856){\line(1,0){6.5}}\qbezier(7.75,13.4404)(7.75,13.4404)(11,19.0356)\qbezier(14.75,13.4404)(14.75,13.4404)(11.5,19.0356)\qbezier(15.25,13.4404)(15.25,13.4404)(18.5,19.0356)\qbezier(22.25,13.4404)(22.25,13.4404)(19,19.0356)\qbezier(11.5,19.9356)(11.5,19.9356)(14.75,25.5308)\qbezier(18.5,19.9356)(18.5,19.9356)(15.25,25.5308)
\put(30,0){\circle{1}}\put(37.5,0){\circle{1}}\put(45,0){\circle{1}}\put(33.75,6.49519){\circle{1}}\put(41.25,6.49519){\circle{1}}\put(37.5,12.9904){\circle{1}}\put(30.5,0){\line(1,0){6.5}}\put(38,0){\line(1,0){6.5}}\put(34.25,6.49519){\line(1,0){6.5}}\qbezier(30.25,0.45)(30.25,0.45)(33.5,6.04519)\qbezier(37.25,0.45)(37.25,0.45)(34,6.04519)\qbezier(37.75,0.45)(37.75,0.45)(41,6.04519)\qbezier(44.75,0.45)(44.75,0.45)(41.5,6.04519)\qbezier(34,6.94519)(34,6.94519)(37.25,12.5404)\qbezier(41,6.94519)(41,6.94519)(37.75,12.5404)
\put(45,0){\circle{1}}\put(52.5,0){\circle{1}}\put(60,0){\circle{1}}\put(48.75,6.49519){\circle{1}}\put(56.25,6.49519){\circle{1}}\put(52.5,12.9904){\circle{1}}\put(45.5,0){\line(1,0){6.5}}\put(53,0){\line(1,0){6.5}}\put(49.25,6.49519){\line(1,0){6.5}}\qbezier(45.25,0.45)(45.25,0.45)(48.5,6.04519)\qbezier(52.25,0.45)(52.25,0.45)(49,6.04519)\qbezier(52.75,0.45)(52.75,0.45)(56,6.04519)\qbezier(59.75,0.45)(59.75,0.45)(56.5,6.04519)\qbezier(49,6.94519)(49,6.94519)(52.25,12.5404)\qbezier(56,6.94519)(56,6.94519)(52.75,12.5404)
\put(37.5,12.9904){\circle{1}}\put(45,12.9904){\circle{1}}\put(52.5,12.9904){\circle{1}}\put(41.25,19.4856){\circle{1}}\put(48.75,19.4856){\circle{1}}\put(45,25.9808){\circle{1}}\put(38,12.9904){\line(1,0){6.5}}\put(45.5,12.9904){\line(1,0){6.5}}\put(41.75,19.4856){\line(1,0){6.5}}\qbezier(37.75,13.4404)(37.75,13.4404)(41,19.0356)\qbezier(44.75,13.4404)(44.75,13.4404)(41.5,19.0356)\qbezier(45.25,13.4404)(45.25,13.4404)(48.5,19.0356)\qbezier(52.25,13.4404)(52.25,13.4404)(49,19.0356)\qbezier(41.5,19.9356)(41.5,19.9356)(44.75,25.5308)\qbezier(48.5,19.9356)(48.5,19.9356)(45.25,25.5308)
\put(15,25.9808){\circle{1}}\put(22.5,25.9808){\circle{1}}\put(30,25.9808){\circle{1}}\put(18.75,32.476){\circle{1}}\put(26.25,32.476){\circle{1}}\put(22.5,38.9711){\circle{1}}\put(15.5,25.9808){\line(1,0){6.5}}\put(23,25.9808){\line(1,0){6.5}}\put(19.25,32.476){\line(1,0){6.5}}\qbezier(15.25,26.4308)(15.25,26.4308)(18.5,32.026)\qbezier(22.25,26.4308)(22.25,26.4308)(19,32.026)\qbezier(22.75,26.4308)(22.75,26.4308)(26,32.026)\qbezier(29.75,26.4308)(29.75,26.4308)(26.5,32.026)\qbezier(19,32.926)(19,32.926)(22.25,38.5211)\qbezier(26,32.926)(26,32.926)(22.75,38.5211)
\put(30,25.9808){\circle{1}}\put(37.5,25.9808){\circle{1}}\put(45,25.9808){\circle{1}}\put(33.75,32.476){\circle{1}}\put(41.25,32.476){\circle{1}}\put(37.5,38.9711){\circle{1}}\put(30.5,25.9808){\line(1,0){6.5}}\put(38,25.9808){\line(1,0){6.5}}\put(34.25,32.476){\line(1,0){6.5}}\qbezier(30.25,26.4308)(30.25,26.4308)(33.5,32.026)\qbezier(37.25,26.4308)(37.25,26.4308)(34,32.026)\qbezier(37.75,26.4308)(37.75,26.4308)(41,32.026)\qbezier(44.75,26.4308)(44.75,26.4308)(41.5,32.026)\qbezier(34,32.926)(34,32.926)(37.25,38.5211)\qbezier(41,32.926)(41,32.926)(37.75,38.5211)
\put(22.5,38.9711){\circle{1}}\put(30,38.9711){\circle{1}}\put(37.5,38.9711){\circle{1}}\put(26.25,45.4663){\circle{1}}\put(33.75,45.4663){\circle{1}}\put(30,51.9615){\circle{1}}\put(23,38.9711){\line(1,0){6.5}}\put(30.5,38.9711){\line(1,0){6.5}}\put(26.75,45.4663){\line(1,0){6.5}}\qbezier(22.75,39.4211)(22.75,39.4211)(26,45.0163)\qbezier(29.75,39.4211)(29.75,39.4211)(26.5,45.0163)\qbezier(30.25,39.4211)(30.25,39.4211)(33.5,45.0163)\qbezier(37.25,39.4211)(37.25,39.4211)(34,45.0163)\qbezier(26.5,45.9163)(26.5,45.9163)(29.75,51.5115)\qbezier(33.5,45.9163)(33.5,45.9163)(30.25,51.5115)
\put(60,0){\circle{1}}\put(67.5,0){\circle{1}}\put(75,0){\circle{1}}\put(63.75,6.49519){\circle{1}}\put(71.25,6.49519){\circle{1}}\put(67.5,12.9904){\circle{1}}\put(60.5,0){\line(1,0){6.5}}\put(68,0){\line(1,0){6.5}}\put(64.25,6.49519){\line(1,0){6.5}}\qbezier(60.25,0.45)(60.25,0.45)(63.5,6.04519)\qbezier(67.25,0.45)(67.25,0.45)(64,6.04519)\qbezier(67.75,0.45)(67.75,0.45)(71,6.04519)\qbezier(74.75,0.45)(74.75,0.45)(71.5,6.04519)\qbezier(64,6.94519)(64,6.94519)(67.25,12.5404)\qbezier(71,6.94519)(71,6.94519)(67.75,12.5404)
\put(75,0){\circle{1}}\put(82.5,0){\circle{1}}\put(90,0){\circle{1}}\put(78.75,6.49519){\circle{1}}\put(86.25,6.49519){\circle{1}}\put(82.5,12.9904){\circle{1}}\put(75.5,0){\line(1,0){6.5}}\put(83,0){\line(1,0){6.5}}\put(79.25,6.49519){\line(1,0){6.5}}\qbezier(75.25,0.45)(75.25,0.45)(78.5,6.04519)\qbezier(82.25,0.45)(82.25,0.45)(79,6.04519)\qbezier(82.75,0.45)(82.75,0.45)(86,6.04519)\qbezier(89.75,0.45)(89.75,0.45)(86.5,6.04519)\qbezier(79,6.94519)(79,6.94519)(82.25,12.5404)\qbezier(86,6.94519)(86,6.94519)(82.75,12.5404)
\put(67.5,12.9904){\circle{1}}\put(75,12.9904){\circle{1}}\put(82.5,12.9904){\circle{1}}\put(71.25,19.4856){\circle{1}}\put(78.75,19.4856){\circle{1}}\put(75,25.9808){\circle{1}}\put(68,12.9904){\line(1,0){6.5}}\put(75.5,12.9904){\line(1,0){6.5}}\put(71.75,19.4856){\line(1,0){6.5}}\qbezier(67.75,13.4404)(67.75,13.4404)(71,19.0356)\qbezier(74.75,13.4404)(74.75,13.4404)(71.5,19.0356)\qbezier(75.25,13.4404)(75.25,13.4404)(78.5,19.0356)\qbezier(82.25,13.4404)(82.25,13.4404)(79,19.0356)\qbezier(71.5,19.9356)(71.5,19.9356)(74.75,25.5308)\qbezier(78.5,19.9356)(78.5,19.9356)(75.25,25.5308)
\put(90,0){\circle{1}}\put(97.5,0){\circle{1}}\put(105,0){\circle{1}}\put(93.75,6.49519){\circle{1}}\put(101.25,6.49519){\circle{1}}\put(97.5,12.9904){\circle{1}}\put(90.5,0){\line(1,0){6.5}}\put(98,0){\line(1,0){6.5}}\put(94.25,6.49519){\line(1,0){6.5}}\qbezier(90.25,0.45)(90.25,0.45)(93.5,6.04519)\qbezier(97.25,0.45)(97.25,0.45)(94,6.04519)\qbezier(97.75,0.45)(97.75,0.45)(101,6.04519)\qbezier(104.75,0.45)(104.75,0.45)(101.5,6.04519)\qbezier(94,6.94519)(94,6.94519)(97.25,12.5404)\qbezier(101,6.94519)(101,6.94519)(97.75,12.5404)
\put(105,0){\circle{1}}\put(112.5,0){\circle{1}}\put(120,0){\circle{1}}\put(108.75,6.49519){\circle{1}}\put(116.25,6.49519){\circle{1}}\put(112.5,12.9904){\circle{1}}\put(105.5,0){\line(1,0){6.5}}\put(113,0){\line(1,0){6.5}}\put(109.25,6.49519){\line(1,0){6.5}}\qbezier(105.25,0.45)(105.25,0.45)(108.5,6.04519)\qbezier(112.25,0.45)(112.25,0.45)(109,6.04519)\qbezier(112.75,0.45)(112.75,0.45)(116,6.04519)\qbezier(119.75,0.45)(119.75,0.45)(116.5,6.04519)\qbezier(109,6.94519)(109,6.94519)(112.25,12.5404)\qbezier(116,6.94519)(116,6.94519)(112.75,12.5404)
\put(97.5,12.9904){\circle{1}}\put(105,12.9904){\circle{1}}\put(112.5,12.9904){\circle{1}}\put(101.25,19.4856){\circle{1}}\put(108.75,19.4856){\circle{1}}\put(105,25.9808){\circle{1}}\put(98,12.9904){\line(1,0){6.5}}\put(105.5,12.9904){\line(1,0){6.5}}\put(101.75,19.4856){\line(1,0){6.5}}\qbezier(97.75,13.4404)(97.75,13.4404)(101,19.0356)\qbezier(104.75,13.4404)(104.75,13.4404)(101.5,19.0356)\qbezier(105.25,13.4404)(105.25,13.4404)(108.5,19.0356)\qbezier(112.25,13.4404)(112.25,13.4404)(109,19.0356)\qbezier(101.5,19.9356)(101.5,19.9356)(104.75,25.5308)\qbezier(108.5,19.9356)(108.5,19.9356)(105.25,25.5308)
\put(75,25.9808){\circle{1}}\put(82.5,25.9808){\circle{1}}\put(90,25.9808){\circle{1}}\put(78.75,32.476){\circle{1}}\put(86.25,32.476){\circle{1}}\put(82.5,38.9711){\circle{1}}\put(75.5,25.9808){\line(1,0){6.5}}\put(83,25.9808){\line(1,0){6.5}}\put(79.25,32.476){\line(1,0){6.5}}\qbezier(75.25,26.4308)(75.25,26.4308)(78.5,32.026)\qbezier(82.25,26.4308)(82.25,26.4308)(79,32.026)\qbezier(82.75,26.4308)(82.75,26.4308)(86,32.026)\qbezier(89.75,26.4308)(89.75,26.4308)(86.5,32.026)\qbezier(79,32.926)(79,32.926)(82.25,38.5211)\qbezier(86,32.926)(86,32.926)(82.75,38.5211)
\put(90,25.9808){\circle{1}}\put(97.5,25.9808){\circle{1}}\put(105,25.9808){\circle{1}}\put(93.75,32.476){\circle{1}}\put(101.25,32.476){\circle{1}}\put(97.5,38.9711){\circle{1}}\put(90.5,25.9808){\line(1,0){6.5}}\put(98,25.9808){\line(1,0){6.5}}\put(94.25,32.476){\line(1,0){6.5}}\qbezier(90.25,26.4308)(90.25,26.4308)(93.5,32.026)\qbezier(97.25,26.4308)(97.25,26.4308)(94,32.026)\qbezier(97.75,26.4308)(97.75,26.4308)(101,32.026)\qbezier(104.75,26.4308)(104.75,26.4308)(101.5,32.026)\qbezier(94,32.926)(94,32.926)(97.25,38.5211)\qbezier(101,32.926)(101,32.926)(97.75,38.5211)
\put(82.5,38.9711){\circle{1}}\put(90,38.9711){\circle{1}}\put(97.5,38.9711){\circle{1}}\put(86.25,45.4663){\circle{1}}\put(93.75,45.4663){\circle{1}}\put(90,51.9615){\circle{1}}\put(83,38.9711){\line(1,0){6.5}}\put(90.5,38.9711){\line(1,0){6.5}}\put(86.75,45.4663){\line(1,0){6.5}}\qbezier(82.75,39.4211)(82.75,39.4211)(86,45.0163)\qbezier(89.75,39.4211)(89.75,39.4211)(86.5,45.0163)\qbezier(90.25,39.4211)(90.25,39.4211)(93.5,45.0163)\qbezier(97.25,39.4211)(97.25,39.4211)(94,45.0163)\qbezier(86.5,45.9163)(86.5,45.9163)(89.75,51.5115)\qbezier(93.5,45.9163)(93.5,45.9163)(90.25,51.5115)
\put(30,51.9615){\circle{1}}\put(37.5,51.9615){\circle{1}}\put(45,51.9615){\circle{1}}\put(33.75,58.4567){\circle{1}}\put(41.25,58.4567){\circle{1}}\put(37.5,64.9519){\circle{1}}\put(30.5,51.9615){\line(1,0){6.5}}\put(38,51.9615){\line(1,0){6.5}}\put(34.25,58.4567){\line(1,0){6.5}}\qbezier(30.25,52.4115)(30.25,52.4115)(33.5,58.0067)\qbezier(37.25,52.4115)(37.25,52.4115)(34,58.0067)\qbezier(37.75,52.4115)(37.75,52.4115)(41,58.0067)\qbezier(44.75,52.4115)(44.75,52.4115)(41.5,58.0067)\qbezier(34,58.9067)(34,58.9067)(37.25,64.5019)\qbezier(41,58.9067)(41,58.9067)(37.75,64.5019)
\put(45,51.9615){\circle{1}}\put(52.5,51.9615){\circle{1}}\put(60,51.9615){\circle{1}}\put(48.75,58.4567){\circle{1}}\put(56.25,58.4567){\circle{1}}\put(52.5,64.9519){\circle{1}}\put(45.5,51.9615){\line(1,0){6.5}}\put(53,51.9615){\line(1,0){6.5}}\put(49.25,58.4567){\line(1,0){6.5}}\qbezier(45.25,52.4115)(45.25,52.4115)(48.5,58.0067)\qbezier(52.25,52.4115)(52.25,52.4115)(49,58.0067)\qbezier(52.75,52.4115)(52.75,52.4115)(56,58.0067)\qbezier(59.75,52.4115)(59.75,52.4115)(56.5,58.0067)\qbezier(49,58.9067)(49,58.9067)(52.25,64.5019)\qbezier(56,58.9067)(56,58.9067)(52.75,64.5019)
\put(37.5,64.9519){\circle{1}}\put(45,64.9519){\circle{1}}\put(52.5,64.9519){\circle{1}}\put(41.25,71.4471){\circle{1}}\put(48.75,71.4471){\circle{1}}\put(45,77.9423){\circle{1}}\put(38,64.9519){\line(1,0){6.5}}\put(45.5,64.9519){\line(1,0){6.5}}\put(41.75,71.4471){\line(1,0){6.5}}\qbezier(37.75,65.4019)(37.75,65.4019)(41,70.9971)\qbezier(44.75,65.4019)(44.75,65.4019)(41.5,70.9971)\qbezier(45.25,65.4019)(45.25,65.4019)(48.5,70.9971)\qbezier(52.25,65.4019)(52.25,65.4019)(49,70.9971)\qbezier(41.5,71.8971)(41.5,71.8971)(44.75,77.4923)\qbezier(48.5,71.8971)(48.5,71.8971)(45.25,77.4923)
\put(60,51.9615){\circle{1}}\put(67.5,51.9615){\circle{1}}\put(75,51.9615){\circle{1}}\put(63.75,58.4567){\circle{1}}\put(71.25,58.4567){\circle{1}}\put(67.5,64.9519){\circle{1}}\put(60.5,51.9615){\line(1,0){6.5}}\put(68,51.9615){\line(1,0){6.5}}\put(64.25,58.4567){\line(1,0){6.5}}\qbezier(60.25,52.4115)(60.25,52.4115)(63.5,58.0067)\qbezier(67.25,52.4115)(67.25,52.4115)(64,58.0067)\qbezier(67.75,52.4115)(67.75,52.4115)(71,58.0067)\qbezier(74.75,52.4115)(74.75,52.4115)(71.5,58.0067)\qbezier(64,58.9067)(64,58.9067)(67.25,64.5019)\qbezier(71,58.9067)(71,58.9067)(67.75,64.5019)
\put(75,51.9615){\circle{1}}\put(82.5,51.9615){\circle{1}}\put(90,51.9615){\circle{1}}\put(78.75,58.4567){\circle{1}}\put(86.25,58.4567){\circle{1}}\put(82.5,64.9519){\circle{1}}\put(75.5,51.9615){\line(1,0){6.5}}\put(83,51.9615){\line(1,0){6.5}}\put(79.25,58.4567){\line(1,0){6.5}}\qbezier(75.25,52.4115)(75.25,52.4115)(78.5,58.0067)\qbezier(82.25,52.4115)(82.25,52.4115)(79,58.0067)\qbezier(82.75,52.4115)(82.75,52.4115)(86,58.0067)\qbezier(89.75,52.4115)(89.75,52.4115)(86.5,58.0067)\qbezier(79,58.9067)(79,58.9067)(82.25,64.5019)\qbezier(86,58.9067)(86,58.9067)(82.75,64.5019)
\put(67.5,64.9519){\circle{1}}\put(75,64.9519){\circle{1}}\put(82.5,64.9519){\circle{1}}\put(71.25,71.4471){\circle{1}}\put(78.75,71.4471){\circle{1}}\put(75,77.9423){\circle{1}}\put(68,64.9519){\line(1,0){6.5}}\put(75.5,64.9519){\line(1,0){6.5}}\put(71.75,71.4471){\line(1,0){6.5}}\qbezier(67.75,65.4019)(67.75,65.4019)(71,70.9971)\qbezier(74.75,65.4019)(74.75,65.4019)(71.5,70.9971)\qbezier(75.25,65.4019)(75.25,65.4019)(78.5,70.9971)\qbezier(82.25,65.4019)(82.25,65.4019)(79,70.9971)\qbezier(71.5,71.8971)(71.5,71.8971)(74.75,77.4923)\qbezier(78.5,71.8971)(78.5,71.8971)(75.25,77.4923)
\put(45,77.9423){\circle{1}}\put(52.5,77.9423){\circle{1}}\put(60,77.9423){\circle{1}}\put(48.75,84.4375){\circle{1}}\put(56.25,84.4375){\circle{1}}\put(52.5,90.9327){\circle{1}}\put(45.5,77.9423){\line(1,0){6.5}}\put(53,77.9423){\line(1,0){6.5}}\put(49.25,84.4375){\line(1,0){6.5}}\qbezier(45.25,78.3923)(45.25,78.3923)(48.5,83.9875)\qbezier(52.25,78.3923)(52.25,78.3923)(49,83.9875)\qbezier(52.75,78.3923)(52.75,78.3923)(56,83.9875)\qbezier(59.75,78.3923)(59.75,78.3923)(56.5,83.9875)\qbezier(49,84.8875)(49,84.8875)(52.25,90.4827)\qbezier(56,84.8875)(56,84.8875)(52.75,90.4827)
\put(60,77.9423){\circle{1}}\put(67.5,77.9423){\circle{1}}\put(75,77.9423){\circle{1}}\put(63.75,84.4375){\circle{1}}\put(71.25,84.4375){\circle{1}}\put(67.5,90.9327){\circle{1}}\put(60.5,77.9423){\line(1,0){6.5}}\put(68,77.9423){\line(1,0){6.5}}\put(64.25,84.4375){\line(1,0){6.5}}\qbezier(60.25,78.3923)(60.25,78.3923)(63.5,83.9875)\qbezier(67.25,78.3923)(67.25,78.3923)(64,83.9875)\qbezier(67.75,78.3923)(67.75,78.3923)(71,83.9875)\qbezier(74.75,78.3923)(74.75,78.3923)(71.5,83.9875)\qbezier(64,84.8875)(64,84.8875)(67.25,90.4827)\qbezier(71,84.8875)(71,84.8875)(67.75,90.4827)
\put(52.5,90.9327){\circle{1}}\put(60,90.9327){\circle{1}}\put(67.5,90.9327){\circle{1}}\put(56.25,97.4279){\circle{1}}\put(63.75,97.4279){\circle{1}}\put(60,103.923){\circle{1}}\put(53,90.9327){\line(1,0){6.5}}\put(60.5,90.9327){\line(1,0){6.5}}\put(56.75,97.4279){\line(1,0){6.5}}\qbezier(52.75,91.3827)(52.75,91.3827)(56,96.9779)\qbezier(59.75,91.3827)(59.75,91.3827)(56.5,96.9779)\qbezier(60.25,91.3827)(60.25,91.3827)(63.5,96.9779)\qbezier(67.25,91.3827)(67.25,91.3827)(64,96.9779)\qbezier(56.5,97.8779)(56.5,97.8779)(59.75,103.473)\qbezier(63.5,97.8779)(63.5,97.8779)(60.25,103.473)

\put(60,103.923){\circle*{2}}\put(60,90.9327){\circle*{2}}
\put(45,64.9519){\circle*{2}}
\put(75,64.9519){\circle*{2}}

\put(30,38.9711){\circle*{2}}
\put(15,12.9904){\circle*{2}}\put(0,0) {\circle*{2}}
\put(45,12.9904){\circle*{2}}

\put(90,38.9711){\circle*{2}}
\put(75,12.9904){\circle*{2}}
\put(105,12.9904){\circle*{2}}\put(120,0){\circle*{2}}

\end{picture} 

\caption{A largest outer mutual-visibility set of
Sierpi\'nski triangle graph $ST_3^4$} \label{outer_fig}
\end{figure}
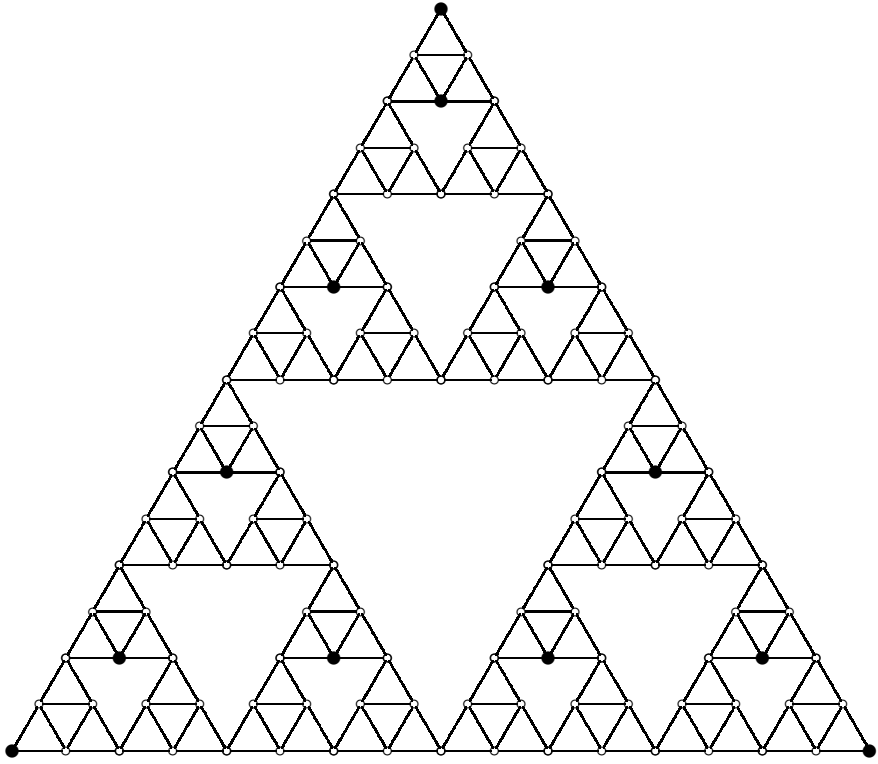

\begin{thm}
If $n \ge 1$, then 
\begin{displaymath}
\mu_o(ST_3^n) =
        \left \{ \begin{array}{llll}
              3,  &  n = 1   \\
              3^{n-2}+3,  &  {\rm otherwise}   \\
             \end{array}. \right.
\end{displaymath}
\end{thm}

\begin{proof}
Remind that we show in Proposition \ref{zgornjaO} that 
$\mu_o(ST_3^n) \le 3^{n-2}+3$. Hence,  
we have to show that there exists an outer mutual-visibility set of $ST_3^n$ with 
$3^{n-2}+3$ vertices for every $n\ge 1$. 
Since it is not difficult to construct  an outer visibility set with 3  vertices of $ST_3^1$, let assume that $n\ge 2$.
We now construct an outer mutual-visibility set $M$ of $ST_3^n$ such that 
for every copy of $ST_3^2$, say $H_2$, in $ST_3^n$  we insert exactly one proper vertex of $H_2$. 
Moreover, we also insert all three extreme vertices 
of $ST_3^n$ in $M$.
Clearly, $M$ contains exactly  $3^{n-2}+3$ vertices. 

We now show that $M$ is an outer mutual-visibility set of $ST_3^n$, i.e., we have to show that 
every $v \in V(ST_3^n)$ and every $u \in M$ are $M$-visible. 
If $u$ and $v$ belong to the same copy of $ST_3^2$, it is not difficult to see that this claim holds. Let then $u$ and $v$  belong to different copies of $ST_3^2$, say $H_2$ and $H_2'$, 
respectively. If none of $H_2$ and $H_2'$ contains an extreme vertex of $ST_3^n$,  then  $M$ is $H_2$- and $H_2'$-outer-proper. By Proposition \ref{properO} it holds that every $u \in V(H_2)$ (resp. $v \in V(H_2')$)
and every extreme vertex $p_i$ of $H_2$ (resp. extreme vertex $p_i'$ of $H'$), $i \in [3]_0$,  are $M$-visible. 
Remind also that every two extreme vertices of a copy of $ST_3^2$
that does contain an extreme vertex of $ST_3^n$
are $M$ visible. 
Since every shortest $u,v$-path contains $p_i$ for some $i \in [3]_0$
and  $p_j'$ for some $j \in [3]_0$, 
it follows that $u$ and $v$ are  $M$-visible. 

If $H_2$ and $H_2'$ contains an extreme vertex of $ST_3^n$, then note that  every $u,v$-path contains $p_i$ for some $i \in [3]_0$ and  $p_j'$ for some $j \in [3]_0$ such
that $p_i$ and $p_j'$ are not extreme vertices of $ST_3^n$. Moreover, 
$u$ and  $p_i$ (resp. $v$ and $p_j'$) are clearly $M$-visible. 
Since it follows that $u$ and $v$ are  $M$-visible, the proof is complete.
\end{proof}

In order to study the total mutual-visibility sets of  Sierpi\'nski  triangle graphs, we need the following result. 

\begin{clm} \label{cl}
If $n\ge 1$ and $u$ is a non-extreme vertex of $ST_3^n$, then 
$ST_3^n$ admits two  vertices, say $v$ and $w$, such that $I(v,w)=\{u\}$.
\end{clm}

\begin{proof}
If  $u$ is non-extreme vertex  of a copy of $V(ST_3^1)$, say $H$, in $V(ST_3^n)$, then let $v$ and $w$ be the extreme vertices of $H$ adjacent to $u$ and the  claim readily follows.   
Otherwise, $u$ is a common vertex of two copies of $V(ST_3^1)$ in $V(ST_3^n)$, say 
 $H$ and $H'$. If $v$ (resp. $w$) is a vertex of $V(H)$ (resp. $V(H')$) adjacent to $u$,  then obviously $u$ is the only vertex on a shortest $v,w$-path. 
\end{proof}

\begin{thm} \label{t_thm}
If $n\ge 1$, then $$\mu_t(ST_3^n) = 3.$$
Moreover,  $M$ is the largest total mutual-visibility set of $ST_3^n$ if and only if $M=X(ST_3^n)$.
\end{thm}

\begin{proof}
Let $M$ be a total mutual-visibility set of $ST_3^n$. 
By Claim \ref{cl}, the vertex $u \in M $ cannot be a non-extreme vertex of  $V(ST_3^n)$, since otherwise there exist two neighbours of that $u$ are not $M$-visible.  
It follows that the largest total mutual-visibility set of  $ST_3^n$ contains all three extreme vertices of  $ST_3^n$. 
\end{proof}

The dual mutual-visibility numbers of  Sierpi\'nski  triangle graphs 
are given in the next result.

\begin{thm} \label{dmv_thm}
If $n\ge 1$, then 
\begin{displaymath}
\mu_d(ST_3^n) =
        \left \{ \begin{array}{llll}
              3,  &  n = 1   \\
              4,  &  {\rm otherwise}   \\
             \end{array}. \right.
\end{displaymath}
\end{thm}

\begin{proof}
Let $M$ be a dual mutual-visibility set of $ST_3^n$. 

If \( n = 1 \), suppose that \( M \) is a set with 4 vertices. Let \( u \) and \( v \) be vertices in \( V(ST_3^1) \setminus M \). Since \( u \) and \( v \) must be adjacent (otherwise, they would not be \( M \)-visible), there must exist two extreme vertices, say \( p_i \) and \( p_j \), in \( M \), along with a vertex \( w \in  M \) adjacent to both \( p_i \) and \( p_j \). Hence, 
\( I(p_i, p_j) = \{ w \} \), which leads to a contradiction.
It follows that \( M \) cannot contain 4 vertices. Since it is easy to verify that the three extreme vertices of \( ST_3^1 \) form a dual mutual-visibility set, this case is settled.

Let \( n \ge 2 \). Consider \( M \) as a dual mutual-visibility set of \( ST_3^n \) and let \( u \in M \). We will first show that \( u \) is either an extreme vertex or a vertex adjacent to an extreme vertex of \( ST_3^n \).

Assume for contradiction that \( u \) is at least distance 2 from every extreme vertex of \( ST_3^n \). Note that \( u \) and any two other vertices in \( M \) cannot induce a triangle, as this would imply the existence of two vertices adjacent to this triangle that are not \( M \)-visible, leading to a contradiction.

If \( u \) is a common vertex of two copies of \( ST_3^1 \), say \( H \) and \( H' \), then we can find vertices \( v, w \not \in M \) such that \( v \in V(H) \) and \( w \in V(H') \). Since \( v \) and \( w \) are not \( M \)-visible, we reach a contradiction.

If \( u \) is a non-extreme vertex of a copy of \( ST_3^1 \), say \( H \), in \( ST_3^n \), then the two extreme vertices of \( H \) adjacent to \( u \) cannot both be in \( M \) or both in \( V(H) \setminus M \) because they would not be \( M \)-visible. Since we have already established that an extreme vertex of \( H \) cannot be in \( M \), it follows that any vertex in \( M \) must either be an extreme vertex of \( ST_3^n \) or a vertex adjacent to an extreme vertex of \( ST_3^n \).

Let \( u \) be a vertex adjacent to an extreme vertex, say \( p_i \), of \( ST_3^n \). Note that there exists a vertex \( w \) adjacent to \( u \) such that \( I(p_i, w) = \{ u \} \). Hence, if a vertex adjacent to an extreme vertex \( p_i \) is in \( M \), then \( p_i \) must also be in \( M \).

Consider a vertex \( z \ne u \) adjacent to an extreme vertex \( p_i \) of \( ST_3^n \). If \( p_i, u, \) and \( z \) are all in \( M \), then the cardinality of \( M \) is 3, since no other extreme vertices of \( ST_3^n \) can be included in \( M \). 

If \( z \not \in M \), let \( p_j \) and \( p_k \) be extreme vertices of \( ST_3^n \) that are not adjacent to \( u \). Without loss of generality, assume \( u \in I(p_i, p_j) \). Clearly, \( p_j \not \in M \). Thus, the cardinality of \( M \) is at most 4.

Let \( v \) be a vertex adjacent to \( p_k \) such that \( u \not \in I(p_i, p_k) \). It is straightforward to verify that \( \{ p_i, u, p_k, v \} \) forms a dual visibility set of \( ST_3^n \). This completes the proof.

\end{proof}

\subsection*{Funding}
This work was supported by the Slovenian Research Agency under the grants P1-0297, J1-2452, J1-1693 and  J2-1731.


\end{document}